\documentclass[11pt]{article}
\usepackage{setspace}
\usepackage{amssymb}
\usepackage{enumerate}
\usepackage{amsfonts}
\usepackage{amsmath}
\usepackage{soul}
\usepackage[utf8]{inputenc}
\usepackage{mathtools}
\usepackage{comment}
\usepackage{tikz}
\newcommand{\mycomment}[1]{}
\usepackage{textcomp}

\def\UseBibLatex{1}

\newcommand{\UsePackage}[1]{%
  \IfFileExists{styles/#1.sty}{%
      \usepackage{styles/#1}%
   }{%
      \IfFileExists{../styles/#1.sty}{%
         \usepackage{../styles/#1}%
      }{%
         \usepackage{#1}%
      }%
   }%
}
\newcommand{\MThanks}[1]{%
   \thanks{Partially supported by NSF grant DMS-2154101
   }%
}

\usepackage[T1]{fontenc}
\usepackage{lmodern}
\usepackage{textcomp}

\usepackage{amsmath}%
\usepackage{amssymb}%

\usepackage{todonotes}
\usepackage[in]{fullpage}%

\usepackage[amsmath,thmmarks]{ntheorem}%
\theoremseparator{.}%

\usepackage{titlesec}%
\titlelabel{\thetitle. }%
\usepackage{mleftright}%
\usepackage{xspace}%
\usepackage{graphicx}
\usepackage{hyperref}%

\usepackage{hyperref}%
\hypersetup{%
      unicode,
    breaklinks,%
   colorlinks=true,%
      urlcolor=[rgb]{0.25,0.0,0.0},%
      linkcolor=[rgb]{0.5,0.0,0.0},%
      citecolor=[rgb]{0,0.2,0.445},%
      filecolor=[rgb]{0,0,0.4},
      anchorcolor=[rgb]={0.0,0.1,0.2}%
}
\usepackage[ocgcolorlinks]{ocgx2}

\theoremseparator{.}%

\theoremstyle{plain}%
\newtheorem{theorem}{Theorem}[section]

\newtheorem{lemma}[theorem]{Lemma}

\newtheorem{corollary}[theorem]{Corollary}
\newtheorem{claim}[theorem]{Claim}%

\newtheorem{proposition}[theorem]{Proposition}

\theoremstyle{plain}%
 \theorembodyfont{\upshape}%
 \newtheorem{construction}[theorem]{Construction}
\newtheorem*{remark:unnumbered}[theorem]{Remark}%
\newtheorem{remark}[theorem]{Remark}%
\newtheorem{definition}[theorem]{Definition}

\newtheorem{example}[theorem]{Example}

\newtheorem{fact}[theorem]{Fact}

\newcommand{\myqedsymbol}{\rule{2mm}{2mm}}

\theoremheaderfont{\em}%
\theorembodyfont{\upshape}%
\theoremstyle{nonumberplain}%
\theoremseparator{}%
\theoremsymbol{\myqedsymbol}%
\newtheorem{proof}{Proof:}%

\theoremsymbol{\myqedsymbol}%
\providecommand{\emphind}[1]{}%
\renewcommand{\emphind}[1]{\emph{#1}\index{#1}}

\definecolor{blue25emph}{rgb}{0, 0, 11}

\providecommand{\emphic}[2]{}
\renewcommand{\emphic}[2]{\textcolor{blue25emph}{%
      \textbf{\emph{#1}}}\index{#2}}

\providecommand{\emphi}[1]{}%
\renewcommand{\emphi}[1]{\emphic{#1}{#1}}

\definecolor{almostblack}{rgb}{0, 0, 0.3}

\providecommand{\emphw}[1]{}%
\renewcommand{\emphw}[1]{{\textcolor{almostblack}{\emph{#1}}}}%

\providecommand{\emphOnly}[1]{}%
\renewcommand{\emphOnly}[1]{\emph{\textcolor{blue25}{\textbf{#1}}}}

\newcommand{\HLink}[2]{\hyperref[#2]{#1~\ref*{#2}}}
\newcommand{\HLinkSuffix}[3]{\hyperref[#2]{#1\ref*{#2}{#3}}}

\newcommand{\thmlab}[1]{{\label{theo:#1}}}
\newcommand{\thmref}[1]{\HLink{Theorem}{theo:#1}}

\newcommand{\corlab}[1]{\label{cor:#1}}
\newcommand{\corref}[1]{\HLink{Corollary}{cor:#1}}%

\providecommand{\deflab}[1]{}
\renewcommand{\deflab}[1]{\label{def:#1}}
\newcommand{\defref}[1]{\HLink{Definition}{def:#1}}

\providecommand{\eqlab}[1]{}%
\renewcommand{\eqlab}[1]{\label{equation:#1}}

\newcommand{\remove}[1]{}%

\usepackage[inline]{enumitem}

\newlist{compactenumA}{enumerate}{5}%
\setlist[compactenumA]{topsep=0pt,itemsep=-1ex,partopsep=1ex,parsep=1ex,%
   label=(\Alph*)}%

\newlist{compactenuma}{enumerate}{5}%
\setlist[compactenuma]{topsep=0pt,itemsep=-1ex,partopsep=1ex,parsep=1ex,%
   label=(\alph*)}%

\newlist{compactenumI}{enumerate}{5}%
\setlist[compactenumI]{topsep=0pt,itemsep=-1ex,partopsep=1ex,parsep=1ex,%
   label=(\Roman*)}%

\newlist{compactenumi}{enumerate}{5}%
\setlist[compactenumi]{topsep=0pt,itemsep=-1ex,partopsep=1ex,parsep=1ex,%
   label=(\roman*)}%

\newlist{compactitem}{itemize}{5}%
\setlist[compactitem]{topsep=0pt,itemsep=-1ex,partopsep=1ex,parsep=1ex,%
   label=\ensuremath{\bullet}}%

\providecommand{\BibLatexMode}[1]{}
\providecommand{\BibTexMode}[1]{}

\ifx\UseBibLatex\undefined%
  \renewcommand{\BibLatexMode}[1]{}
  \renewcommand{\BibTexMode}[1]{#1}
\else
  \renewcommand{\BibLatexMode}[1]{#1}
  \renewcommand{\BibTexMode}[1]{}
\fi

\BibLatexMode{%
   \usepackage[bibencoding=utf8,style=alphabetic,backend=bibtex]{biblatex}%
   \UsePackage{my_biblatex}%
}

\numberwithin{figure}{section}%
\numberwithin{table}{section}%
\numberwithin{equation}{section}%

\BibLatexMode{%
   \bibliography{template}
}
\usepackage{cleveref}
\begin{document}

\title{Merges of Smooth Classes and Their Properties}

\author{Morgan Bryant\MThanks\\\\University of Maryland}
\maketitle

\begin{abstract}
Given two Fraïssé-like classes with generic limits, we ask whether we can merge the two classes into one class with a generic limit. We study the properties of these merges and their generics, as well as their connections to structural Ramsey theory and the Hrushovski property (EPPA).
\end{abstract}

\section{Introduction}
Given a countable, relational language $\mathcal{L}$, a pair $(K,\leq)$ is a \textit{smooth class} if $K$ is a class of finite $\mathcal{L}$-structures whose elements are related by $\leq$, which is determined by universal formulas (see \S2). \textit{Fraïssé classes} are smooth classes whose relation is simply $\subseteq$, and these Fraïssé classes are the classes used in Fraïssé's original constructions. The main motivation for this work, and many other works on smooth classes, comes from classes of Shelah-Spencer sparse graphs (see Example \ref{hrushovskigraph}). These particular classes have served as counterexamples in several model theoretic settings, and have proved to have some interesting properties. The results of this paper were inspired by observations made when studying these classes of graphs. 

Fraïssé limits/generics are $\omega$-categorical, ultrahomogeneous structures with respect to a Fraïssé class, and they are well-understood, classical objects. The generics (also called limits) of proper smooth classes, on the other hand, are far more oblique and less free objects. The goal of this paper is to better understand general smooth classes and their generics in both the strictly model-theoretic sense and in applications of model theory. 

In \S2, we will study certain expansions of smooth classes and their generics by way of \textit{merging smooth classes}. This is done by taking two smooth classes $K_1$ and $K_2$ in the languages $\mathcal{L}_1$ and $\mathcal{L}_2$ and forming a new smooth class, $K_1\circledast K_2$, called the \textit{merge} of $K_1$ and $K_2$, in the language $\mathcal{L}_1\cup \mathcal{L}_2$ such that for $i=1,2$ and all $A\in K_1\circledast K_2$, $A|_{\mathcal{L}_i} \in K_i$.  In some cases, $K_1\circledast K_2$ has a generic $M^*$.  

When $K_1$ and $K_2$ are closed under substructure, it is known by \cite{Evans_2019} that for $i=1,2$, $M^*|_{\mathcal{L}_i}\cong  M_i$, where $M_i$ is the generic of $K_i$. We prove that this can be strengthened by assuming that each of $K_1$ and $K_2$ additionally have \textit{parallel strongness} and \textit{smooth intersections} (defined in \S2.1.1).  In this case,  we show that for any $C^*\subseteq M^*$ an infinite set definable in $M^*$ by an existential $\mathcal{L}_2$-formula, $C^*|_{\mathcal{L}_1} \cong M_1$. We consider this to be the main theorem from the study of merging smooth classes, and it has many consequences, as described in \S 2.3.1. In particular, this theorem provides information about the generics of the original classes $K_1$ and $K_2$.

In \S2.3 we study which model theoretic properties are transferred from the generics $M_1$ and $M_2$ to $M^*$. By strengthening a theorem in \cite{KL}, we show that if $M_1$ and $M_2$ are both atomic, then $M^*$ is atomic. This does not hold for saturation; a merge of classes of Shelah-Spencer graphs gives a counterexample to this transfer. 

\bigskip
In \S3, we turn to structural Ramsey theory and study smooth classes and their merges in this context. Merges of Fraïssé classes have been the quintessential examples of classes with the \textit{Ramsey property} (defined in \S3). In \cite{kechris2004fraisselimitsramseytheory}, for example, Fraïssé classes are often merged with the Fraïssé class of all finite linear orders as a way to "rigidify" the classes. The KPT correspondence, which is the main result of \cite{kechris2004fraisselimitsramseytheory}, gives an equivalence between Fraïssé classes of rigid structures with the Ramsey property and their generics having an extremely amenable automorphism group. This correspondence has drummed up much of the recent interest in classes with the Ramsey property. The correspondence holds for smooth classes of rigid structures, as shown in \cite{ghadernezhad2016automorphismgroupsgenericstructures}.

We are interested in which smooth classes and their merges have the Ramsey property. This is motivated by \cite{bodirsky2012newramseyclassesold}, in which it is shown that whenever $K_1$ and $K_2$ are Fraïssé classes of rigid structures with the Ramsey property, then $K_1\circledast K_2$ has the Ramsey property. It remains unclear as to when this result extends to smooth classes. An intermediate step in the direction of this goal is studying which classes and their merges have the \textit{Hrushovski property} (often written as EPPA) (defined in \S3.1). 

While it is known by \cite{8d041ed3-5ed1-3e56-9a99-3f73c69d33f4} that any Fraïssé class with the \textit{free amalgamation property} has EPPA, (\cite{Evans_2019}, \cite{ghadernezhad2016automorphismgroupsgenericstructures}) have shown that this does not extend to smooth classes. In \S3.2, we give an explicit family of smooth classes with the free amalgamation property and prove these classes and their merges have EPPA. This expands upon the ideas given for a specific smooth class in (\cite{Evans_2019},\cite{Evans_2021}) and uses the strongest existing EPPA results from \cite{hubička2022eppaclassesstrengtheningsherwiglascar}. We also put a smooth relation on the class of infinitely many equivalence relations given by \cite{https://doi.org/10.1002/malq.201400036} and generalize the argument in \cite{https://doi.org/10.1002/malq.201400036} to obtain a non-trivial smooth class without the free amalgamation property which has EPPA.

\section{Merging Smooth Classes}
\subsection{Background}
In this section, we will give the background needed to understand many of the results in this paper. Smooth classes, which also go under the name of "strong classes", are a generalization of Fraïssé's original constructions. The definition of these classes can vary from author to author. We choose to use one of the most general versions of the definition of smooth classes for the most general results. 
\subsubsection{Smooth Classes}
For the entirety of this paper, we will consider only countable, relational languages. This is done out of convenience; many of the constructions in this section and in section \ref{mergessection} still apply when the languages contain functions. 
We begin with recalling the following definition from \cite{KL}:
\begin{definition}
    Let $\mathcal{L}$ be a relational language. Let $K$ be a class of finite $\mathcal{L}$-structures which is closed under isomorphism. We say the pair $(K,\leq)$, where $\leq$ is a binary relation on the elements of $K$, is a \textit{smooth class} if 
    \begin{enumerate}
        \item $\leq$ is transitive and for all $A$, $B \in K$, $A\leq B \Rightarrow A\subseteq B$
        \item For all $B\in K$ and every enumeration $\overline{b}:=(b_1,\dots, b_n)$ of $B$, there is a (possibly infinite) set $\Phi_{\overline{b}}(\overline{x})$  of universal $\mathcal{L}$-formulas with $|\overline{x}| = n$ such that $\Delta_{\overline{b}}\subseteq\Phi_{\overline{b}}(\overline{x})$, where $\Delta_{\overline{b}}$ is the atomic diagram of $B$ according to the enumeration $\overline{b}$, such that for any $C\in K$ with $B\subseteq C$, $$B\leq C \Leftrightarrow C\models \Phi_{\overline{b}}(\overline{b}) $$
        \item  For $A, B\in K$, if $\overline{a} = (a_1,\dots, a_n)$ is an enumeration of $A$ and there exists an isomorphism $f:A\rightarrow B$, then for the enumeration $\overline{b} = (f(a_1),\dots, f(a_n))$ of $B$, we require that $\Phi_{\overline{a}} = \Phi_{\overline{b}}$. 
        \item $\emptyset \in K$ and for every $A\in K$, $\emptyset\leq A$
    \end{enumerate}
\end{definition}
Clearly, any class containing $\emptyset$ and equipped with $\subseteq$ as its relation is a smooth class. We now give an example of a smooth class whose relation is not $\subseteq$.  This particular class will appear in several parts of this paper:
\begin{example}
\label{hrushovskigraph}
    Let $\alpha \in (0,1)$, and $\mathcal{L}_\alpha$ a finite, relational language. For any $\mathcal{L}_\alpha$-structure $A$ on which every relation $R\in \mathcal{L}_\alpha$ is symmetric and irreflexive, define the \textit{dimension function} $\delta_\alpha$ on $A$ as $$\delta_\alpha(A) := |A| - \alpha \cdot \sum_{R\in \mathcal{L}} N_R(A) $$
    where $N_R(A)$ is the number of distinct subsets of $A$ on which $R$ holds. 
    Define the class $(K_\alpha, \leq_\alpha)$ as $$K_\alpha = \{A: \delta_\alpha(A') \geq 0 \text{ for all } A'\subseteq A\}$$
    and set $$A\leq_{\alpha} B \Leftrightarrow \delta_\alpha(B') - \delta_\alpha(A) \geq 0 \text{ for all } B' \text{ such that } A\subseteq B' \subseteq B$$ 
    $(K_\alpha, \leq_\alpha)$ is a well known smooth class and is called the class of Shelah-Spencer $\alpha$-graphs in $\mathcal{L}_\alpha$. 
    \end{example}
Smooth classes are meant to be a generalization of classes used in Fraïssé's constructions. In particular, these classes may also have limits as in Fraïssé constructions. To see this, we will need to define the usual properties of a class with respect to the relation $\leq$ as opposed to $\subseteq$:
\begin{definition}
Let $(K,\leq)$ be a smooth class.
\begin{enumerate}
    \item $(K,\leq)$ has the \textit{amalgamation property} (AP) if, for any $A,B,C\in K$ such that there exist embeddings $h_1:A\rightarrow C$ and $h_2: A\rightarrow B$ where $h_1(A)\leq C$ and $h_2(A)\leq C$, then there exists some $D\in K$ such that there exists embeddings $f:B\rightarrow D$ and $g:C\rightarrow D$ with $f(B)\leq D$, $g(C)\leq D$, and $g\circ h_1=f\circ h_2$. 

    We say $(K,\leq)$ has \textit{disjoint amalgamation} (dAP) if we can find a $D\in K$ and embeddings $f:B\rightarrow D$, $g:C\rightarrow D$ with $f(B)\leq D$, $g(C)\leq D$, $f\circ h_2=g\circ h_1$, and $f(B)\cap g(C) = f\circ h_2(A)$.
    \item For any $A,B,C\in K$ with $A\subseteq B, C$, the structure $C*_A B$ is defined as the structure with universe $B\cup C$ where the only relations between $B$ and $C$ are contained in $A$. We say $(K,\leq)$ has \textit{free amalgamation} (fAP) if $A,B,C\in K$ and $A\leq B,C$, then $B,C\leq C*_A B$ and $C*_A B\in K$.   
    \item $(K,\leq)$ has \textit{parallel strongness} (PS) if for any $A,B,C\in K$ such that $A\subseteq C$ and $A\leq B$, there exists some $D\in K$ and embeddings $f:C\rightarrow D$ and $g:B\rightarrow D$ such that $f(C)\leq D$, $g(B)\subseteq D$,  and $g(A) = f(A)$.  $(K,\leq)$ has \textit{disjoint parallel strongness} (dPS) if this can be done so that $f(C)\cap g(B) = g(A)=f(A)$. \\
\textbf{Note:} This property is sometimes referred to as \textit{full amalgamation} in the literature. (e.g., in \cite{BALDWIN19961})
    \item $(K,\leq)$ has \textit{smooth intersections} if for every $A,B,C\in K$,  $A\leq B \Rightarrow A\cap C \leq B\cap C$\\
\textbf{Note:} Some authors assume this property as part of the definition of a smooth class. (e.g., in \cite{BALDWIN19961}, \cite{ghadernezhad2016automorphismgroupsgenericstructures})

\end{enumerate}
\end{definition}
Only AP is required for a generalization of Fraïssé's theorem; dPS, fAP, and dAP will have greater roles in the next sections.  It is worth noting that the classes in Example \ref{hrushovskigraph} have dPS, fAP, and smooth intersections, as shown in \cite{BALDWIN19961}, and we will use this fact throughout this paper.  

We call a smooth class $(K,\leq)$ a \textit{Fraïssé class} if for $A,B\in K$, $A\leq B \Leftrightarrow A\subseteq B$, $K$ has AP, and is closed under substructure. These classes have traditional Fraïssé limits, and they vacuously have smooth intersections and PS.  
\bigskip

We must now define limits, or generics, for any smooth class, and then give conditions for the existence of a generic of a smooth class.  

\bigskip
Suppose $(K,\leq)$ is a smooth class. Given an $\mathcal{L}$-structure $M$, for $A\in K$ with $A\subseteq M$, we write $A\leq M$ if and only if for any enumeration $\overline{a}$ of $A$, $M\models \Phi_{\overline{a}}(\overline{a})$.  
\begin{definition}
    \deflab{generic}
\bigskip
    We say an $\mathcal{L}$-structure $M$ is a \textit{generic} for $(K,\leq)$ if 
    \begin{enumerate}
        \item There exist $\{A_i\}_{i\in\omega}$ with $A_i \in K$ such that $M = \bigcup_{i\in\omega} A_i$ and $A_i\leq A_j$ when $i\leq j$. 
        \item If $A\leq M$ and $ A\leq B$,  for $A,B\in K$, then there is an embedding $f: B \rightarrow M$ such that $f|_{A} = Id_A$ and $f(B)\leq M$. 
    \end{enumerate}
\end{definition}

This is essentially a word-for-word generalization of a Fraïssé limit. The following holds essentially by Fraïssé's original proofs:
\begin{proposition}
    \thmlab{smoothgeneric}
    A smooth class $(K,\leq)$ has a generic $M$ if and only if $(K,\leq)$ has AP and $K$ contains only countably many isomorphism types. Moreover, any two generics of a smooth class $(K,\leq)$ are isomorphic. 
\end{proposition}

As noted before, the study of the generics of smooth classes will be a central throughout this paper. In the next section, we will discuss the idea of combining smooth classes and studying the resulting expansions of their generics. 

\subsection{Merges}
\label{mergessection}
Our main motivation in the study of merges was to study certain, arguably natural, expansions of smooth classes, particularly the classes of  Shelah-Spencer graphs. For any $\alpha\in (0,1)$, the generic $M_\alpha$ of the Shelah-Spencer class $(K_\alpha,\leq_\alpha)$ (see \ref{hrushovskigraph}) is well understood and its theory has a concrete $\forall_2$-axiomatization given by (\cite{axiomatization},\cite{gunatilleka2018theoriesbaldwinshihypergraphsatomic}). In studying expansions of these graphs, it became increasingly clear that certain expansions lead to the discovery of interesting properties of the generic $M_\alpha$. This launched an investigation that led to the study of \textit{merges of smooth classes}, which will be defined in this section. 
\subsubsection{The Set-Up}
We now turn to the problem of combining smooth classes. The idea of combining Fraïssé classes is not new; in fact, it is a rather natural method of constructing new Fraïssé classes. This section can be thought of as a generalization of combining Fraïssé classes, but there are a number of technicalities brought on by the smooth relations that we must deal with.

\begin{definition}
    Fix languages $\mathcal{L}_i$, $i\in \omega$ where $\mathcal{L}_i\cap \mathcal{L}_j =\emptyset \Leftrightarrow i\neq j$. Fix a family of smooth classes $\{(K_i,\leq_i): i\in\omega\}$ such that each $(K_i,\leq_i)$ is a class of $\mathcal{L}_i$-structures. We define a new smooth class $(K^*,\leq_*)$ in the language $\mathcal{L}^*:= \bigcup_{i\in I} \mathcal{L}_i$ as follows: $$K^* = \{A: A|_{\mathcal{L}_i} \in K_i\;\forall i \in I\}$$ and for $A,B\in K^*$, $$A\leq_*B \Leftrightarrow A|_{\mathcal{L}_i} \leq_i B|_{\mathcal{L}_i} \;\; \forall i\in I$$
     We call the class $(K^*,\leq_*)$ defined above the \textit{merge} of $\{(K_i,\leq_i):i\in I\}$. 
\end{definition}
\textbf{Notation:} For smooth classes $(K_1,\leq_1)$ and $(K_2,\leq_2)$, we will write $K_1 \circledast K_2$ to denote the merge of the classes under the relation $\leq_*$. For the set of classes $\{(K_i,\leq_i):i\in I\}$, we will write $\circledast_{i\in I} K_i$ to denote the merge of these classes under the relation $\leq_*$. 

\bigskip
It is clear from the definition of smooth classes that the merge of smooth classes will be itself a smooth class. Let $(K^*,\leq_*)$ be the merge of the family $\{(K_i,\leq_i): i\in I\}$. Let $A\in K^*$ and let $\overline{a}$ be an enumeration of $A$. For $i\in I$, let $\Phi^i_{\overline{a}}$ denote the set of universal $\mathcal{L}_i$-formula such that for all $B\in K_i$, $B\models \Phi^i_{\overline{a}}(\overline{a}) \Leftrightarrow A|_{\mathcal{L}_i}\leq_*B$. Let $\Phi^*_{\overline{a}}$ denote the set of universal $\mathcal{L}^*$-formulas such that for all $B\in K^*$, $B\models \Phi^*_{\overline{a}}(\overline{a}) \Leftrightarrow A\leq_*B$. Notice that $\Phi^*_{\overline{a}}$ can be chosen of the form $\bigcup_{i\in I}\Phi^i_{\overline{a}}$.

\bigskip
We now give sufficient conditions for the existence of a generic of the merge. The conditions below can be much more flexible based on which classes one is attempting to merge.  Essentially, we need to prove that the merge has AP.  To get this, we must have that each class in the merge "agrees" on a cardinality of an amalgamation.  We also need that the merged class does indeed have infinitely many elements. 

This is essentially what the below definition ensures:

\bigskip
\textbf{Notation:} For a smooth class $(K,\leq)$, define $$C_{(K,\leq)} :=\{n: A\in K, \;\; |A|=n\}$$
\begin{definition}
    For languages $\mathcal{L}_i$, $i\in \omega$ where $\mathcal{L}_i\cap \mathcal{L}_j =\emptyset \Leftrightarrow i\neq j$, smooth classes $(K_i,\leq_i)$ which are classes of $\mathcal{L}_i$-structures, and $I\subseteq \omega$, we say the set of smooth classes $\{(K_i,\leq_i):i\in I\}$ has \textit{uniform dAP} if:
    \begin{enumerate}
        \item $I\subseteq \omega$ and for all $i, j\in I$, 
        $$C_{(K_i,\leq_i)} = C_{(K_j, \leq_j)}$$
        \item For any $n,m,k \in C_{(K_i, \leq_i)}$, there exists some $j\in C_{(K_i,\leq_i)}$ such that for \textbf{any} $i\in I$, and any $A_i$, $B_i$, $C_i \in K_i$ where $|A_i| = n$, $|B_i| = m$, $|C_i| = k$ and $A_i\leq_i B_i, C_i$, then there exists a disjoint  amalgam $D_i\in K_i$ of $B_i$ and $C_i$ over $A_i$ with $A_i, B_i, C_i \leq_i D_i$ and $|D_i| = j$.

    \end{enumerate}
\end{definition}

\begin{remark}
When all classes in $\mathcal{S}:=\{(K_i,\leq_i): i\in I\}$ with $I\subseteq \omega$  have closure under substructure and dAP, $\mathcal{S}$ has uniform dAP trivially. Because of this and \thmref{mergeexistencethm} we will frequently work within this framework in the context of merging classes.   
\end{remark}
\begin{theorem}
\thmlab{mergeexistencethm}If a set of smooth classes $\mathcal{S}: = \{(K_i,\leq_i): i\in I\}$ has uniform dAP, then the merge $(K^*,\leq_*)$ has a generic in $\mathcal{L}^*$. 
\end{theorem}
\begin{proof}
We must prove, by \thmref{smoothgeneric}, that $(K^*,\leq_*)$ has AP. We will in fact prove that it has dAP.
 For each $i\in I$ and $A\in K^*$ denote $A|_{\mathcal{L}_i}$ by $A_i$. Let $A,B, C \in K^*$ such that there exist embeddings $\alpha: A\rightarrow C$ and $\beta: A\rightarrow B$ for which $\alpha(A) \leq C$ and $\beta(A) \leq B$. For each $i\in I$, let $\alpha_i$ and $\beta_i$ denote the $\mathcal{L}_i$-embeddings $\alpha$ and $\beta$ restricted to the language $\mathcal{L}_i$. 
Because $\mathcal{S}$ has uniform dAP, for each $i\in I$, we may choose $D_i \in K_i$ such that for all $i,j\in I$, $|D_i| = |D_j|$, and for every $i\in I$, there exist $\mathcal{L}_i$-embeddings $f_i: C_i \rightarrow _i D_i$ and $g_i: B_i \rightarrow D_i$ with $f(C_i)\leq_i D_i$, $g(B_i)\leq_i D_i$,  $f_i\circ \alpha_i = g_i \circ \beta_i$, and $f(C_i)\cap f(B_i)= f_i\circ \alpha_i(A_i)$. 
Enumerate $A$ as $\{a_1,\dots, a_t\}$, $B-\beta(A)$ as $\{b_1,\dots, b_n\}$, and $C - \alpha(A)$ as $\{c_1,\dots, c_m\}$. Enumerate each $D_i$ as $$D_i = \{c^i_1,\dots, c^i_m\}\cup \{b^i_1,\dots, b^i_n\} \cup \{a^i_1,\dots, a_t^i\} \cup \{e^i_1,\dots, e^i_j\}$$ where $c^i_k = f_i(c_k)$,  $b^i_k = g_i(b_k)$, and $a_k^i = f_i\circ \alpha_i(a_k)$.
Take a set $N$ with $|N| = |D_i|$ and enumerate $$N = \{w_1,\dots, w_m\}\cup \{d_1,\dots, d_n\} \cup \{s_1,\dots, s_t\}\cup \{e_1,\dots, e_j\}$$ In each language $\mathcal{L}_i$, we will define $N|_{\mathcal{L}_i}$ so that the function $h_i: N \rightarrow D_i$ for $i\in I$ defined $h_i(w_k) = c^i_k$, $h_i(d_k) = b^i_k$, $h_i(s_k) = a_k^i$ and $h_i(e_k) = e^i_k$ is an $\mathcal{L}_i$-isomorphism. It follows that $N\in K^*$ by definition. 

Moreover, $G: B\rightarrow N$ defined as $G(b_k) = d_k$ and $G(\beta(a_k)) = s_k$ and $F: C\rightarrow N$ defined as $F(c_k) = w_k $ and $F(\alpha(a_k)) = s_k$ are both $\mathcal{L}^*$-embeddings by the construction of $N$. It is also clear that $G(B)\leq_* N$ and $F(C)\leq_* N$. Thus, $K^*$ has dAP. 
\end{proof}
Similar arguments to the proof above can show that if every class in $\mathcal{S}$ has dPS and/or fAP, then so does the merge $(K^*,\leq_*)$.

\begin{remark}
The assumptions in \thmref{smoothgeneric} can be loosened, but the assumption that each class has dAP is nearly necessary. Without each class having dAP, the possible merged classes with generics are often nearly trivial and highly constrainted, depending on how "strict" the AP in each class is. For example, let the class $(K,\leq)$ be such that $K$ is the class of all finite linear orders and $A\leq B$ if and only if $A$ is an initial segment of $B$. This class has AP, but certainly not dAP. The only classes $(K_2,\leq_2)$ which can be merged with $(K,\leq)$ and produce a generic of the merge are the classes $(K_2,\leq_2)$ for which $\leq_2$ is $\subseteq$ and every relation in the language holds (or, equivalently, does not hold) on every tuple in every structure of $K_2$. 
\end{remark}
\subsection{Properties of Merges}
\subsubsection{Preserving Generics in the Merge}
Ideally, when we merge classes and obtain a generic of the merged classes, we would like for the generic of the merged classes to be an expansion of each of the original generics. To have this, we must have that the merged class does not "miss" any information from the original classes. A way of ensuring this is to assume the classes are closed under substructure. In many natural applications, this is a reasonable assumption. 

We now prove that this holds under some additional assumptions:

\begin{proposition}
    \label{firstmerge}
    Suppose $(K_1,\leq_1)$ and $(K_2,\leq_2)$ are smooth classes in $\mathcal{L}_1$ and $\mathcal{L}_2$ respectively. Assume that both classes are closed under substructure, and both have dAP and dPS. Let $M^*$ be the generic of the merge $(K^*,\leq_*)$ of $K_1$ and $K_2$. Then for $i=1,2$, $M^*|_{\mathcal{L}_i}\cong M_i$, where $M_i$ is the generic of $(K_i,\leq_i)$.
\end{proposition}
\begin{proof}
By symmetry, we need only prove the theorem for $i=1$. 
We must prove that $M^*|_{\mathcal{L}_1}$ satisfies conditions (1) and (2) from \defref{generic}. For $\mathcal{L}^*$-structures $A,B\in K^*$, we write $A\leq_1 B$ if and only if $A|_{\mathcal{L}_1} \leq_1 B|_{\mathcal{L}_1}$. It is easy to see that, since $M^*$ is the generic of $K^*$, $M^*|_{\mathcal{L}_1}$ satisfies (1). It remains to show (2). 

Let $A\leq_1 M^*$ and suppose $A\leq_1 B_1$ for some $B_1\in K_1$. As $M^*$ is the generic of $(K^*,\leq_*)$, we can write $M^* = \bigcup_{n\in\omega} C_n$ where $C_n\leq_* C_{n+1}$ and $C_n\in K^*$. Thus, there is some $n$ for which $A\leq_1 C_n$. By taking an isomorphic copy of $B_1$, we may assume $C_n\cap B_1 = A$. \\
By applying dAP and closure under substructure in $\mathcal{L}_1$, we can find some $D_1\in K_1$ with universe $(C_n\cup B_1)$ such that $D_1$ is a disjoint amalgam over $A$ in $\mathcal{L}_1$ with $C_n, B_1 \leq_1 D_1$. Define $B_2$ so that the universe of $B_2$ is the same as that of $B_1$, and define the $\mathcal{L}_2$-structure of $B_2$ so that $B_2\in K_2$ with $B_2\geq A$ and $B_2\cap C_n = A$ (in the $\mathcal{L}_2$ sense).

Now, in the language $\mathcal{L}_2$, using dPS over $A$, we can find an $\mathcal{L}_2$-structure $D_2\in K_2$ with universe  $C_n\cup B_2 \in K_2$ so that $C_n \leq_2 D_2$, $ B_2\subseteq_2 D_2$, and $|D_1| = |D_2|$.
Define structures $D, B\in K^*$ to be so that $D|_{\mathcal{L}_i} = D_i$, $B|_{\mathcal{L}_i} = B_i$ for $i=1,2$. 

Notice this then implies that $C_n \leq_* D $ by definition of $\leq_*$. By the genericity of $M^*$, there exists some $\mathcal{L}^*$-embedding $g: D\rightarrow M^*$ which is the identity on $C_n$ such that $D^*:= g(D) \leq_* M^*$ and $C_n\leq_* D^*$. Now, $g(B)\leq_1 D^*$ and $A\leq_1 g(B)$. Moreover, $g(B)|_{\mathcal{L}_1} \cong B_1$. By transitivity, $g(B) \leq_1 M^*$. This proves (2). 
\end{proof}

When classes are only assumed to be closed under substructure, it is indeed true that the generic of the merge is an expansion of the original generics. This follows from a set-up and proof given in \S2 of \cite{Evans_2019}, though one must convert their theorems to the language and ideas of merging smooth classes. 
Whenever smooth classes $(K_1,\leq_1)$ and $(K_2,\leq_2)$ are each closed under substructure with dAP,  the merge $K_1 \circledast K_2$ is a \textit{strong extension} of both $K_1$ and $K_2$ in the sense defined in \S 2 of \cite{Evans_2019}. Because $K_1\circledast K_2$ is a strong extension, it is proven that if $M^*$ is the generic of $K_1\circledast K_2$, then $M^*|_{\mathcal{L}_i}\cong M_i$ where $M_i$ is the generic of $(K_i,\leq_i)$.  

The proof of this result from \cite{Evans_2019} and the proof of Proposition \ref{firstmerge} are entirely different. It turns out that the additional assumption that the classes both have dPS gives an even stronger result than Proposition \ref{firstmerge}, which will be shown in \thmref{mainthm}. This was inspired from behavior observed from the merge of a class of Shelah-Spencer  graphs (Example \ref{hrushovskigraph})  with the Fraïssé class of finite equivalence relations. We observed a relationship between equivalence classes in the merge of the generic and the original generic of the class of Shelah-Spencer graphs. The assumption that classes have dPS is modeled after the fact that classes of Shelah-Spencer graphs have dPS.

\bigskip
We will now state and prove \thmref{mainthm}, which we consider to be the main result in our study of merging smooth classes. 

\bigskip
\textbf{Notation:} If $\leq_1$ is a relation on a class in $\mathcal{L}_1$ and $A,B$ are structures in a language $\mathcal{L}\supseteq \mathcal{L}_1$, we write $A\leq_1 B$ for $A|_{\mathcal{L}_1} \leq_1 B|_{\mathcal{L}_1}$. If $A|_{\mathcal{L}_1} \cong B|_{\mathcal{L}_1}$ in $\mathcal{L}_1$, we write $A \cong_{\mathcal{L}_1} B$

\begin{theorem}
\thmlab{mainthm}
 Let $(K_1,\leq_1)$ and $(K_2,\leq_2)$ be smooth classes in languages $\mathcal{L}_1$ and $\mathcal{L}_2$ with generics $M_1$ and $M_2$ respectively. Assume both classes are closed under substructure and have dAP. Suppose $K_1$ has smooth intersections and dPS. Let $M^*$ denote the generic of the merge $(K^*,\leq_*)$ of $K_1$ and $K_2$. Suppose $C^* = \varphi(M^*, \overline{m})$ is infinite, where $\overline{m}\subseteq M^*$ and $\varphi$ is an existential formula in $\mathcal{L}_2$. Then, $C^*|_{\mathcal{L}_1}$ is isomorphic to the generic of $(K_1,\leq_1)$ 
\end{theorem}
\begin{proof}
We must verify that $C^*|_{\mathcal{L}_1}$ satisfies (1) and (2) from \defref{generic}. 

Because $M^*$ is the generic of $(K^*,\leq_*)$, we may write $M^* = \bigcup_{n\in \omega} W_n$ where $W_n\leq_* W_{n+1}$ and $W_n\in K^*$ for all $n\in \omega$. By assumption, and definition of $\leq_*$, $$C^* = \bigcup_{n\in \omega} (W_n\cap C^*)
\text{ with } W_n\cap C^* \leq_1 W_{n+1}\cap C^* \text{ for all } n\in \omega$$
By the assumption that all classes are closed under substructure, $W_n\cap C^*|_{\mathcal{L}_1} \in (K_1, \leq_1)$. Thus, $C^*|_{\mathcal{L}_1}$  satisfies (1).

\bigskip

It remains to show that $C^*|_{\mathcal{L}_1}$ satisfies (2). Suppose $A\leq_1 C^*$ and $A\leq_1 B_1$ where $B_1\in K_1$. We may take $B_1$ so that $B_1\cap M^* = A$. Write $B_1-A = \{b_1,\dots, b_k\}$. As $C^*$ infinite, we may find a set of $k$ elements $E:= \{e_1,\dots, e_k\} \subseteq C^*$. Note that $E\in K^*$. For each $i\leq k$, there exists a finite tuple $\overline{h}_i\subseteq M^*$ which witnesses that the existential $\varphi(e_i, \overline{m})$ holds in $M^*$. Let $$V := \{m: m\in \overline{m}\}\cup \{h: h\in \overline{h}_i \text{ for some } i\}$$
We may find some finite $W_n = \{w_1,\dots, w_\ell\}$ for which $A\cup V\subseteq W_n$ and $W_n\leq_* M^*$. Note that $A\leq_1 C^*\cap W_n$ by assumption. Using dAP in $K_1$, we can find an $\mathcal{L}_1$-structure $B'$ with universe $(W_n \cap C^*) \cup B_1$ that is an amalgamation over $A$ so that $A, B_1, W_n\cap C^* \leq_1 B'$. 
Using dPS in $K_1$, we can find an $\mathcal{L}_1$-structure $D_1$ with universe $W_n \cup B'$ that is a dPS amalgam over $W_n\cap C^*$ such that $B'\subseteq D_1$ and $W_n \leq_1 D_1$. Write $D_1 = \{d_1^1,\dots, d_k^1\}\cup \{w_1,\dots, w_\ell\}$ where $d_i^1 = b_i$ for $i\leq k$ and $w_i = w_i\in W_n$. 

Note that $E\cup W_n\subseteq M^*$, so $E \cup W_n\in K^*$, and, moreover, $|E\cup W_n | = |D_1|$. Now, in $\mathcal{L}_2$, define a $\mathcal{L}_2$-structure $D_2$ enumerated $\{d^2_1,\dots, d^2_k\}\cup \{w_i,\dots, w_\ell\}$ so that $W_n\subseteq D_2$ and the map $f: D_2 \rightarrow E\cup W_n$ is such that $f(d^2_i) = e_i$ and $f(w_i) = w_i$ is an $\mathcal{L}_2$-isomorphism. As $W_n\leq_* M^*$, we know that $W_n \leq_2 D_2$. 

Define the $\mathcal{L}^*$-structure $D =\{d_1,\dots, d_k\}\cup\{w_1,\dots, w_\ell\}$ where $w_i\in W_n$ so that $W_n \subseteq D$ (in the $\mathcal{L}^*$ sense); the map $h: D\rightarrow D_2$ defined $h(d_i) = d^2_i$, $h(w_i) = w_i$ is an $\mathcal{L}_2$-isomorphism; and the map $p: D\rightarrow D_1$ defined $p(d_i) = d^1_i$, $p(w_i) = w_i$ is a $\mathcal{L}_1$-isomorphism. Then, notice that $W_n\leq_* D$. 
Now, because $W_n\leq_* M^*$ and $M^*$ is the $K^*$ generic, then there is a $\mathcal{L}^*$-embedding $\alpha: D\rightarrow M^*$ so that $\alpha|_{W_n}$ is the identity map on $W_n$ and $D^*:= \alpha(D) \leq_* M^*$. Let $B^* := \alpha(\{d_1,\dots, d_k\})\cup A$. There is a natural $\mathcal{L}_1$-isomorphism $\rho: B_1 \rightarrow B^*$ by $\rho(b_i) = d_i$, fixing $A$. 

Notice that $B^*\leq_1 \alpha(B')$ (where we are looking at the copy of the $\mathcal{L}_1$-structure $B'$ in $D$). We can write $B' = (W_n\cap C^*)\cup B$, and notice that in $D$, $D\models \varphi(d_i,\overline{m})$ for each $i\leq k$ since $\varphi$ is a $\mathcal{L}_2$ formula and the construction of $D$ preserved and copied the $\mathcal{L}_2$ the relations between each $e_i$, $\overline{m}$ and $h_i$ for $i\leq k$.
Thus, $\alpha(B')\subseteq C^*$, as $W_n\cap C^*$ is fixed by $\alpha$ and as $\varphi(x,\overline{m})$ is existential, $M^* \models \varphi(\alpha(d_i),\overline{m})$, thus ensuring $\alpha(d_i) \in C^*$. Moreover, $D^*\cap C^* = \alpha(W_n\cup B')\cap C^* = (W_n\cap C^*)\cup \alpha(B') = \alpha(B')$. Therefore, $$\alpha(\rho(B_1)) \leq_1 \alpha(B') = D^*\cap C^* \leq_1 C^*$$
Finally, recall that $A\subseteq W_n$, so $\alpha \circ \rho$ is an embedding over $A$ as required. 
This completes the proof of (2). 
\end{proof}
By essentially the same argument, we get
\begin{theorem}
     \label{infinitemerge}
Suppose $\mathcal{S}=\{(K_i,\leq_i): i\in I\}$ is a set of smooth classes  with uniform dAP such that each class $(K_i,\leq_i)$ is closed under substructure and has dPS.  Suppose $M^*$ is the generic of the merge $\circledast_{i\in I} (K_i,\leq_i)$ in the language $\mathcal{L}^* = \{\mathcal{L}_i:i\in I\}$. For every $J\subseteq I$, if $C^* = \varphi(M^*,\overline{m})$ is infinite with $\overline{m}\subseteq M^*$ and $\varphi$ an existential formula in the language $\bigcup_{i\in I-J} \mathcal{L}_i$, then $C^*\cong M_i$, the generic of $(K_i,\leq_i)$,  for $i\in J$.  
\end{theorem}
\thmref{mainthm} has the following direct consequences.

\bigskip
Suppose $(K_1,\leq_1)$ and $(K_2,\leq_2)$ are smooth classes closed under substructure with dAP, and that $(K_1,\leq_1)$ has smooth intersections and dPS. Let $M_1$ and $M_2$ denote their respective generics, and $M^*$ to be the generic of $K_1 \circledast K_2$. 
\begin{corollary}
$Aut(M^*) \cong Aut(M_1)\cap Aut(M_2)$ as subgroups of $Sym(M^*)$.  
\end{corollary}
\begin{corollary}
 If $a\in M_1$, then $M_1 \cong M_1 - \{a\}$.
\end{corollary}
\begin{proof}
    Using \thmref{mainthm} and merging $(K_1,\leq_1)$ with any Fraïssé class $(K_2,\leq_2)$, for any $a\in M^*$, the set $C^* = \{b\in M^*: b\neq a\} $ is indeed isomorphic to $M_1$. 
\end{proof}
The next corollary holds by taking $(K_2,\leq_2)$ to be the class of finite structures with a binary equivalence relation and applying the main theorem:
\begin{corollary}
    There is an expansion of $M_1$ by a binary equivalence relation $E$ such that for any $n\in \omega$, and $\{P_i\}_{i\leq n}$ a set of $n$ equivalence classes of $M_1$, $\bigcup_{i\leq n} P_i \cong M_1$ 
\end{corollary}
The class $(K_1,\leq_1)$ can be taken to be a class of Shelah-Spencer graphs $(K_\alpha, \leq_\alpha)$ with generic $M_\alpha$, and the class $(K_2,\leq)$ can be taken to be the Fraïssé class of all finite linear orders, giving:  
\begin{corollary}
There exists a family of sets $\{W_i\}_{i\in \omega}$ of subsets of $M_\alpha$ such that any finite, non-empty intersection of sets in $\{W_i\}_{i\in\omega}$ is isomorphic to $M_\alpha$.  
\end{corollary}
Take $(K_1, \leq_1)$ to be a class of Shelah-Spencer graphs $(K_\alpha, \leq_\alpha)$ in the language $\mathcal{L}_\alpha := \{E_\alpha\}$, where $E_\alpha$ is binary, and $(K_2,\leq_2)$ to be another class of Shelah-Spencer graphs $(K_\beta, \leq_\beta)$ in the language $\mathcal{L}_\beta := \{E_\beta\}$ where $E_\beta$ is a binary relation. Let $E_\beta\in \mathcal{L}_\beta$. Applying \thmref{mainthm}, the following holds:
\begin{corollary}
 For $a\in M^*$ and $C^*:=\{b\in M^*: \exists x(  E_\beta(a,x) \land E_\beta(x,b) \land \neg E_\beta(a,b))\}$,  $C^*|_{\mathcal{L}_\alpha}\cong M_\alpha$.  
\end{corollary}
All of these, of course, can be extended to results with other Fraïssé classes such as the class of finite graphs, the class of $K_n$ free graphs, etc. 

\bigskip
These results may be thought of as well-behaved expansions of generics. They also give an idea of the overall structure of a generic, reminiscent of the "self-referential" properties of some Fraïssé classes. For example, the rationals under the standard linear order, $(\mathbb{Q}, \leq)$, is the Fraïssé limit of the class of all finite linear orders, and it has the property that every open interval in $(\mathbb{Q}, \leq)$ is indeed isomorphic to $(\mathbb{Q}, \leq)$.  Here, we get that for a sufficiently nice class $(K_1,\leq_1)$, its generic $M_1$ has the property that it can be linearly ordered so that every interval is isomorphic to $M_1$. 

Given \thmref{mainthm}, we see that in many cases there is a relationship between the generic of a merge and the generics of the original smooth class. The question is then what other properties of the generics of the original classes transfer to the generic of the merge. The remainder of the paper will discuss this transfer of properties. 

\subsubsection{Atomic Generics and Merges}\label{atomic}
In this section, we will study smooth classes with atomic generics and their merged generics. It turns out that the property of the generic of a smooth class being atomic is dependent on the universal formula definition of the relation of the smooth class. Because of this dependency, the property of the generic being atomic transfers very well to the generic of a merge. 

We begin by proving the converse of Proposition 3.4 in \cite{KL}, characterizing atomic generics of smooth classes. We will include both directions of the proof for convenience, and because we will need both directions.

Recall: By definition of smoothness, for every $A\in K$ and enumeration $\overline{a}$ of $A$, there is a set of universal formulas $\Phi_{\overline{a}}$ such that 
$$A\leq C \Leftrightarrow C\models \phi(\overline{a}) \text{ for all }\phi\in \Phi_{\overline{a}}$$
\begin{proposition}
\label{atomicchar}
    Let $(K, \leq)$ be a smooth class with a generic $M$. Then, $M$ is atomic if and only if for every $A\in K$ and every enumeration $\overline{a}$ of $A$,  $\Phi_{\overline{a}}$ is finitely generated modulo $Th(M)$. I.e., there exists a finite subset $\Phi^0_{\overline{a}} \subseteq \Phi_{\overline{a}}$ such that for $\psi_{\overline{a}}(\overline{x}):= \bigwedge_{\phi\in \Phi^0_{\overline{a}}} \phi(\overline{x})$,  $M\models \forall \overline{x} (\psi_{\overline{a}}(\overline{x}) \rightarrow \Phi_{\overline{a}}(\overline{x}))$. 
\end{proposition}
\begin{proof}
$(\Rightarrow)$:
Let $A \in K$ and fix an enumeration $\overline{a}=(a_1,\dots, a_n)$ of $A$, and suppose that $M\models \Phi_{\overline{a}}(\overline{a})$. As $M$ is atomic, suppose $\theta(\overline{x})$ isolates $tp_M(\overline{a})$. 
Suppose now that $B\subseteq M$ is enumerated by $\overline{b}= (b_1,\dots, b_n)$ and $M\models \Phi_{\overline{a}}(\overline{b})$. We immediately have that $f: A\rightarrow B$ defined $f(a_i) = b_i$ is an isomorphism. Then $B\in K$, and we get that $\Phi_{\overline{a}} = \Phi_{\overline{b}}$ by definition. Thus, as $M\models \Phi_{\overline{b}}(\overline{b})$,  $B\leq M$. By the genericity of $M$, the isomorphism $f: A\rightarrow B$ extends to an automorphism of $M$. This implies that $M\models \theta(b_1,\dots, b_n)$. Thus, $$M\models \forall \overline{x}(\Phi_{\overline{a}}(\overline{x}) \rightarrow \theta(\overline{x}))$$ 
But, this implies that some finite subset of $\Phi_{\overline{a}}$ must imply $\theta$ modulo $Th(M)$, and hence $\Phi_{\overline{a}}$ itself. 

\bigskip

$(\Leftarrow)$: Let $A\subseteq M$ be enumerated by $\overline{a} = (a_1,\dots, a_n)$. Let $A\subseteq B$ where $B$ is the smallest superstructure of $A$ with $B\leq M$. Enumerate $B$ as $\overline{a}^\smallfrown \overline{b}$, where $\overline{b} =(b_1,\dots, b_m)$. Let $\psi_{\overline{a}^\smallfrown \overline{b}}$ be the formula in the assumption associated with $B$ and $\overline{a}^\smallfrown\overline{b}$. We have that $M\models \psi_{\overline{a}^\smallfrown\overline{b}}(\overline{a}^\smallfrown\overline{b})$. Set $$\theta(y_1,\dots, y_n) := \exists x_1,\dots ,x_m ( \psi_{\overline{a}^\smallfrown\overline{b}}(y_1,\dots, y_n, x_1,\dots, x_m))$$
\textbf{Claim:} $\theta(\overline{x})$ isolates $tp_M(\overline{a})$.\\
Suppose $M\models \theta(\overline{d})$. Then, this implies the existence of some $C$ enumerated $\overline{c} = \overline{d}^\smallfrown(c_1,\dots, c_m)$ where $M\models \psi_{\overline{a}^\smallfrown\overline{b}}(\overline{c})$. We know by assumption that there is an isomorphism $f: B\rightarrow C$ with $f(a_i)= d_i$ and $f(b_i) = c_i$. It follows that $C\leq M$.  By the genericity of $M$, $f$ has an extension to an automorphism of $M$. Thus, $tp(\overline{a}) = tp(\overline{d})$, and it follows that $\theta$ isolates $\overline{a}$.
\end{proof}
\begin{corollary}
\corlab{atomicmerge}
    Let $(K_1,\leq_1)$ and $(K_2,\leq_2)$ be smooth classes in respective languages $\mathcal{L}_1$, $\mathcal{L}_2$ and generics $M_1$, $M_2$. Suppose the merge $(K_1\circledast K_2, \leq_*)$ has a generic $M^*$. Moreover, assume that $M^*|_{\mathcal{L}_1} \cong M_1$ and $M^*|_{\mathcal{L}_2} \cong M_2$.
    If $M_1$ and $M_2$ are atomic then the merged generic $M^*$ is atomic. 
\end{corollary}
\begin{proof}
 For simplicity, we will assume $M^*|_{\mathcal{L}_i} = M_i$ for $i=1,2$. 
 Given some $A\in K^*$, write $A_i$ for $A|_{\mathcal{L}_i}$. Let $\overline{a}$ be an enumeration of $A$. Let $\psi_{\overline{a}}^{1}$ and $\psi^2_{\overline{a}}$ be the $\mathcal{L}_1$- and $\mathcal{L}_2$- formulas given by Proposition \ref{atomicchar} such that for $i=1,2$, 
 $$M_i\models \forall \overline{x} (\psi^i_{\overline{a}}(\overline{x})\leftrightarrow \Phi^i_{\overline{a}}(\overline{x}))$$

 Now, suppose $M^*\models \Phi^*_{\overline{a}}(\overline{b})$ for some $\overline{b}\subseteq M^*$. For each $i=1,2$, this immediately implies that $M^* \models \psi^i_{\overline{a}}(\overline{b})$. Now, assume that $M^*\models \psi^1_{\overline{a}}(\overline{b})\land \psi^{2}_{\overline{a}}(\overline{b})$. Then, $M_i \models \psi^{i}_{\overline{a}}(\overline{b})$, and thus $M_i \models \Phi^i_{\overline{a}}(\overline{b})$. But this implies that $M^* \models  \Phi^*_{\overline{a}}(\overline{b})$. Therefore, $M^*\models  \forall\overline{ x}([\psi^1_{\overline{a}}(\overline{x})\land \psi^{2}_{\overline{a}}(\overline{x})] \Leftrightarrow \Phi^*_{\overline{a}}(\overline{x}))$. We can freely impose that $\psi^{1}_{\overline{a}} \land \psi^{2}_{\overline{a}} \in \Phi_A$. By Proposition \ref{atomicchar}, $M^*$ is atomic. 
\end{proof}
From \corref{atomicmerge} we get another immediate result:
\begin{proposition}
\label{categorical}
    Let $(K_1,\subseteq)$ be a Fraïssé class and $(K_2,\leq_2)$ a smooth class. Suppose the merge $(K_1\circledast K_2, \leq_*)$ has a generic $M^*$ such that for $i=1,2$, $M^*|_{\mathcal{L}_i} \cong M_i$, where $M_i$ is the generic of $K_i$. 
    If $M_2$ is $\omega$-categorical, then $M^*$ is $\omega$-categorical. 
\end{proposition}
\begin{proof}
We will show that every countable structure $N$ such that $N\models Th(M^*)$, $N\cong M^*$. 
Suppose $N\models Th(M^*)$. We must prove that $N$ has properties (1) and (2) from \defref{generic}. 

\bigskip
To prove (1), note that we have that $N\models Th(M_1)\cup Th(M_2)$. By each of $Th(M_1)$ and $Th(M_2)$ being $\omega$-categorical, $N|_{\mathcal{L}_i}\cong M_i$ for $i=1,2$. In particular, we can write $N|_{\mathcal{L}_2} = \bigcup_{n\in\omega} A_n$ with $A_n\leq_2 A_{n+1}$. It follows by $(K_1,\subseteq)$ being a Fraïssé  class that $A_n\leq_* A_{n+1}$ in the full structure $N$ and $A_n \in K^*$. 

It remains to prove that $N$ has property (2). Suppose $A\leq_* N$ and $A\leq_* B$ for some $B\in K_1\circledast K_2$. Let $\overline{a} = \{a_1,\dots, a_n\}$ be an enumeration of $A$ and let $\overline{b} = \overline{a}^\smallfrown\{b_1,\dots, b_m\}$ be an enumeration of $B$. By \corref{atomicmerge}, $M^*$ is atomic, and so by Proposition \ref{atomicchar} there exist $\psi^*_{\overline{a}}$ and $\psi^*_{\overline{b}}$ such that $M^* \models \forall \overline{x}(\psi^*_{\overline{a}}(\overline{x}) \leftrightarrow \Phi^*_{\overline{a}}(\overline{x}))$ and $M^*\models \forall \overline{y} (\psi^*_{\overline{b}}(\overline{y})\leftrightarrow \Phi^*_{\overline{b}}(\overline{y}))$. 

As $M^*$ is a generic, we have that $$M^* \models \forall(x_1,\dots, x_n)( \psi^*_{\overline{a}}(x_1,\dots, x_n) \rightarrow \exists(y_1,\dots,y_m)( \psi^*_{\overline{b}}(x_1,\dots, x_n; y_1,\dots, y_m)))$$ Thus, as $N \models Th(M^*)$, there is an embedding of $f: B\rightarrow N $ such that $f|_A = id_A$ and $f(B)\leq_* N$. This shows that $N$ has property (3). 
\end{proof}  
In the next section, we will prove that there are classes with saturated generics for which the merged generic is not saturated. Similar arguments as in the next section show that Proposition \ref{categorical} does not necessarily extend when both classes are smooth but not Fraïssé. In the proof of Proposition \ref{categorical}, the proof that every countable $N\models Th(M^*)$ has property (2) always holds for $\omega$-categorical smooth classes (not necessarily Fraïssé) $(K_1,\leq_1)$, $(K_2,\leq_2)$; it is the proof that any such $N$ has property (1) that may fail.

\subsubsection{Saturated generics and merges}\label{saturation}
We now ask how the merged generic behaves when both of the original generics are saturated. In contrast to atomic generics, the property of saturation transfers poorly to the generic of a merge. This is due, in part, by the relation $\leq_*$ in the merged generic being dependent the relations on the original classes; the result is that the generic of the merge cannot satisfy all the new types created. 

To make this concrete, we will recall some results and definitions about saturation and smooth classes. For a full treatment of these results, see \cite{BALDWIN19961}.  
\begin{definition}
    Suppose $(K,\leq)$ is a smooth class of finite structures. We say for $A,B\in K$, $(A,B)$ is a \textit{ $\leq$-minimal pair }if for every $A\subseteq B'\subsetneq B$, $A\leq B'$ but $A\not\leq B$.  $(K,\leq)$ has an \textit{infinite chain of $\leq$-minimal pairs} if there is a chain $$A_0\subseteq A_1\subseteq \dots A_n \subseteq \dots$$
such that $A_i\in K$ and $(A_i,A_{i+1})$ is a $\leq$-minimal pair for $i\in \omega$. 
\end{definition}
We reprove a well-known result about smooth classes and saturation:
\begin{theorem}
\thmlab{baldwinminpair}
If $(K,\leq)$ has a generic $M$, smooth intersections, and an infinite chain of $\leq$-minimal pairs, then $M$ is not saturated. 
\end{theorem}
\begin{proof}
    Let $\{A_i\}_{i\in \omega}$ be an infinite chain of $\leq$-minimal pairs. Let  $k = |A_0|$ and let $\Delta_i(x_1,\dots, x_k,\overline{y})$ denote the elementary diagram of $A_i$ over $A_0$. Set $$\Gamma_i(x_1,\dots, x_k) = (\Delta_0(x_1,\dots, x_k) \land \exists \overline{y} (\Delta_i(x_1,\dots, x_k,\overline{y}))$$
    Let $\Gamma^n(x_1,\dots, x_k) = \{\Gamma_i(x_1,\dots, x_k): i\leq n\}$. Note that $\Gamma^n$ is realized in $M$ by a strong embedding of $A_n$ into $M$. Let  $\Gamma:= \bigcup_{n\in\omega}\Gamma^n$.  It follows that $Th(M)\cup \Gamma$ is finitely satisfiable, and hence satisfiable by compactness. However, $M$ does not realize $\Gamma$.  To see this, suppose $M\models \Gamma(A^0)$ for some $A^0\subseteq M$. There is some finite $B\subseteq M$ such that $A^0\subseteq B\leq M$. Suppose $n -1 = |B|$. As $M\models \Gamma(A^0)$, there exist sets $A^i$ such that $A^0\subseteq A^i \subseteq M$ and $(A^i, A^{i+1})$ are $\leq$-minimal pairs for $i\leq n$. Note that $B\cap A^i \leq A^i$, so if $B\cap A^i \neq A^i$, then $A^0 \leq B\cap A^i$ by definition of a $\leq$-minimal pair .  But this is a contradiction, so $A^i\subseteq B$ for all $i\in \omega$. However, this implies $|B|\geq n$, which is a contradiction.  
\end{proof}
\begin{definition}
    A smooth class $(K,\leq)$  has \textit{many minimal pairs} if there exists some $m\in \omega$ such that for all $A\in K$ with $|A|\geq m$ and for all $n\geq 1$, there exists some $A' \in K$ such that $A\leq A'$, $|A'| = |A|+n$, and there is some $B\in K$ for which $(A',B)$ is a $\leq$-minimal pair
\end{definition}
\begin{example}
Note that for $\alpha\in \mathbb{Q}\cap (0,1)$, the class of Shelah-Spencer $\alpha$-graphs $(K_\alpha,\leq_\alpha)$ has many minimal pairs and the generic $M_\alpha$ of $(K_\alpha,\leq_\alpha)$ is saturated by \cite{gunatilleka2018theoriesbaldwinshihypergraphsatomic}. \thmref{genminpairs} will show that for $\alpha$, $\beta \in \mathbb{Q}\cap (0,1)$, the generic $M^*$ of the merge $(K_\alpha \circledast K_\beta, \leq_*)$ is not saturated, despite the fact that $M^*|_{\mathcal{L}_\alpha} \cong M_\alpha$ and $M^*|_{\mathcal{L}_\beta}\cong M_\beta$. In fact, $Th(M^*)$ is not small. This shows that there exist classes with saturated generics whose merged generic is not  saturated, even when both classes satisfy the assumptions of Proposition \ref{firstmerge}.
\end{example}

\begin{theorem}
\thmlab{genminpairs}
Suppose $(K_1,\leq_1)$ and $(K_2,\leq_2)$ are smooth classes, each closed under substructure in the resp. languages $\mathcal{L}_1$, $\mathcal{L}_2$.  Suppose both classes have many minimal pairs, and suppose the merge $(K^*,\leq_*)$ has a generic $M^*$. Then $M^*$ is not saturated. 
\end{theorem}
\begin{proof}
We will show $(K^*,\leq_*)$ has an infinite chain $\{A_i\}_{i\in\omega}$ of $\leq_*$-minimal pairs, and get the conclusion by applying \thmref{baldwinminpair}.
By assumption, there exists some $m\in \omega$ such that for $i=1,2$, if $A\in K_i$ with $|A|\geq m$, then there is some $A' \in K_i$ such that $A\leq_i A'$ and there exists some $B\in K_i$ such that $(A',B)$ is a $\leq_i$-minimal pair. 

Choose some $A_0\in K^*$ such that $|A_0|\geq m$. There exists some $B_1 \in K_1$ so that $(A_0|_{\mathcal{L}_1}, B_1)$ forms a $\leq_1$-minimal pair. On the same universe as $B_1$, by assumption, we can define $B_2\in K_2$ so that $A_0\leq_2 B_2$ and there exists some $C\in K_2$ so that $(B_2,C)$ forms a $\leq_2$-minimal pair. Define $A_1$ on the universe of $B_1$ so that $A_1|_{\mathcal{L}_i} = B_i$ for $i=1,2$. Then, $A_0\leq_2 A_1$ and $(A_0|_{\mathcal{L}_1}, A_1)$ is a $\leq_1$-minimal pair. Thus, $(A_0,A_1)$ is a $\leq_*$-minimal pair. 
    
Suppose $n$ is odd and $A_n$ has been defined so that there exists some $C_2\in K_2$ for which $(A_n|_{\mathcal{L}_2}, C_2)$ is a $\leq_2$-minimal pair. By assumption, we can find some $C_1 \in K_1$ with the same universe as $C_2$ for which $A_n\leq_1 C_1$ and there exists some $D_1\in K_1$ for which $(C_1,D_1)$ forms a $\leq_1$-minimal pair. Let $A_{n+1}$ be defined on the same universe as $C_2$ so that $A_{n+1}|_{\mathcal{L}_i} = C_i$ for $i=1,2$. Then, $A_n\leq_1 A_{n+1}$ and $(A_n, A_{n+1})$ forms a $\leq_2$-minimal pair. Then, $(A_n, A_{n+1})$ is a $\leq_*$-minimal pair. For $n$ even, we switch the roles of $\mathcal{L}_1$ and $\mathcal{L}_2$ in the odd case. 

Carrying out this construction, we have an infinite chain $\{A_i\}_{i\in \omega}$ of $\leq_*$-minimal pairs. 
\end{proof}
\section{Structural Ramsey Theory, EPPA Classes, and Merges}\label{eppa}
We now turn our attention to the branch of study of structural Ramsey theory. For many years, there has been a campaign to identify classes with the EPPA and Ramsey properties (defined below), especially after the links to topological dynamics were discovered. Most work has been done in the context of Fraïssé classes and, in particular, classes of rigid structures. It is notable that merges of Fraïssé classes have shown up as many of the examples of classes with the Ramsey property and as a way to "rigidify" classes. 

The current goal is to try to expand the ideas used on Fraïssé classes to gain an understanding of smooth classes and their merges and which smooth classes have the Ramsey properties. The Ramsey property is hard to achieve in this context; however, a stepping stone is the Hrushovski/EPPA property. In this section we will expand and generalize on previous work to get some classes which do indeed have the EPPA property and for which the EPPA property transfers nicely to some merges of these classes.   
\subsection{Background \& Motivation}
Merges of Fraïssé classes have played an important role in the development of structural Ramsey theory. 
We first recall the definition of the Ramsey property in the smooth class setting:
\begin{definition}
    Suppose $(K, \leq)$ is a smooth class. Let $\binom{B}{A} =\{C\leq B: C\cong A\}$. We say $K$ has the \textit{Ramsey property} if for any $A\leq B$ with $A,B\in K$, there exists some $C\in K$ for which $B\leq C$ for every coloring $c: \binom{C}{A} \rightarrow \{0,1\}$, there is some $B'\in \binom{C}{B}$ such that $c|_{\binom{B'}{A}}$ is constant. 
\end{definition}
It is important to note that we are taking the Ramsey property to be on colorings of \textit{structures} as opposed to colorings of \textit{embeddings}. This is mostly to avoid ambiguity and to include classes of structures which are not rigid.

\bigskip
Much of the focus of the study of Ramsey classes is on classes of rigid structures because of their connection to extreme amenability. The so-called "KPT-Correspondence" coming from the main results of \cite{kechris2004fraisselimitsramseytheory} (and generalized by \cite{ghadernezhad2016automorphismgroupsgenericstructures}) asserts that if a smooth class $(K,\leq)$ is a class of rigid structures with generic $M$, then $Aut(M)$ is extremely amenable if and only if $(K,\leq)$ has the Ramsey property.

Many Fraïssé classes which were proved to have the Ramsey property are indeed merges of Fraïssé classes. Let $(K_{LO},\subseteq)$ denote the Fraïssé class of all finite linear orders. When $(K_F,\subseteq)$ is a Fraïssé class with fAP, by \cite{NESETRIL1983183}, $K_{LO}\circledast K_F$ has the Ramsey Property. This is, however, not always the case for any Fraïssé class merged with $K_{LO}$. For example, the class $(K_E,\subseteq)$ of finite structures with an equivalence relation has the Ramsey property, however, $K_{LO}\circledast  K_E$ does not.

Merges behave better once we are merging classes which are already rigid and known to have the Ramsey property, as shown in the below theorem by \cite{bodirsky2012newramseyclassesold}:
\begin{theorem}
\thmlab{bodirsky}
    If $(K_1, \subseteq)$ and $(K_2,\subseteq)$ are Fraïssé classes of rigid structures (i.e., every automorphism is trivial) and $K_1$, $K_2$ both have the Ramsey property and dAP, then $K_1\circledast K_2$ has the Ramsey Property. Moreover, if $K_3$ is a Fraïssé class of rigid structures with the Ramsey property, and $K_2 \circledast K_{LO}$ has the Ramsey property, then $K_3 \circledast K_2$ has the Ramsey property.
\end{theorem}
This shows, for example, if $K_1$ and $K_2$ are two Fraïssé classes with fAP, $(K_1\circledast K_{LO})*(K_2\circledast K_{LO})$ has the Ramsey property (when we take the languages of each class to be disjoint) and, moreover, so does $K_1 \circledast K_2 \circledast K_{LO}$. It is not yet clear when \thmref{bodirsky} will extend to the non-rigid case, or to the general smooth class setting. The proof of the above theorem relies heavily on $\omega$-categoricity and the results of \cite{kechris2004fraisselimitsramseytheory} in rigid settings. 

It is our current goal to identify smooth classes with the Ramsey property, but this is a difficult task. As many authors have noted,  classes with the Hrushovski property (or EPPA) (defined below) often have connections with classes with the Ramsey property. For example, $K_F$ in the example above has been proven to have EPPA, as well as other classes which have a Ramsey property in some merge. Thus, it seems prudent to study EPPA in the smooth class case in hopes of getting an eventual grasp of the Ramsey property.

We now recall a definition for EPPA in the context of smooth classes:
\begin{definition}
\deflab{eppa}
    Let $(K,\leq)$ be a smooth class. We say $(K,\leq)$ has $\leq$-EPPA if for any $A\in K$, and for any partial automorphisms $f_1,\dots, f_n$ of $A$, where, for each $i$, $dom(f_i)\leq A$ and $range(f_i)\leq A$, then there exists some $B$ with $A\leq B$ such that there exist automorphisms $g_1,\dots, g_n$ of $B$ with $f_i\subseteq g_i$. (When $\leq$ is substructure, we simply say EPPA.)
\end{definition}
There is a clear relationship between classes with EPPA and amenability, shown in \cite{kechris2006turbulenceamalgamationgenericautomorphisms}, which \cite{Evans_2019} realized extended to $\leq$-EPPA:
\begin{proposition}
    Let $(K,\leq)$ be a smooth class with a generic $M$. If $(K,\leq)$ has $\leq$-EPPA, then $Aut(M)$ is amenable.  
\end{proposition}
It is known by \cite{8d041ed3-5ed1-3e56-9a99-3f73c69d33f4} that Fraïssé classes with free amalgamation have EPPA. This result does not extend to smooth classes. The Shelah-Spencer class $(K_\alpha, \leq_\alpha)$ as defined in \ref{hrushovskigraph}, is a class with fAP. It was proven in \cite{ghadernezhad2016automorphismgroupsgenericstructures} and \cite{Evans_2019} that $(K_\alpha, \leq_\alpha)$ has neither the Ramsey property nor $\leq$-EPPA. In fact, they proved that the automorphism group is not amenable for any $\alpha\in (0,1)$. 

\bigskip
We also will make use of a definition of EPPA for permorphisms in the next few sections. These will be definitions originating from Herwig in \cite{Herwig1995ExtendingPI} with which we will use the results from \cite{hubička2022eppaclassesstrengtheningsherwiglascar}. 
\begin{definition}
    Let $\Gamma_{\mathcal{L}} = \{\chi_1,\dots, \chi_n\}$ be a finite group of permutations of the language $\mathcal{L}$ which preserve arity of the relations and fix equivalence. 

    Given an $\mathcal{L}$-structure $A$ and a bijective function $f:A\rightarrow A$, $f$ is a $\chi_i$-permorphism of $A$ if $A\models R(a_1,\dots, a_n) \Leftrightarrow A\models \chi_i(R)(f(a_1),\dots, f(a_n))$ for all relations $R\in \mathcal{L}$. 
    \end{definition}

    \begin{definition}
    A class $(K,\leq)$ has $\Gamma_{\mathcal{L}}$-$\leq$-EPPA if whenever $A\in K$ and $f_1,\dots, f_n$ are partial permorphisms of $A$ such that $f_i$ is a $\chi_i$-permorphism for some $\chi_i \in \Gamma_\mathcal{L}$ and $dom(f_i), range(f_i) \leq A$ for all $i$, then there exists some $B\in K$ such that $A\leq B$ and there exist full permorphisms $g_1,\dots, g_n$ of $B$ such that $g_i$ is a $\chi_i$-permorphism and $g_i \supseteq f_i$. 
\end{definition}

\subsection{Smooth Classes with $\leq$-EPPA}
In this section, we will prove a particular type of smooth class with fAP has $\leq$-EPPA. This will be a generalization on work done in both \cite{Evans_2019} and \cite{hubička2022eppaclassesstrengtheningsherwiglascar}.  We may also extend this to some classes of the same type which do not have fAP. To do this, we will use a nonstandard notion of structure described in \cite{hubička2022eppaclassesstrengtheningsherwiglascar}. 
\subsubsection{Extended Structures}
Because many of the current theorems that give EPPA results involve classes which are Fraïssé classes, or, at minimum, where the relation on the classes is $\subseteq$, they cannot be directly applied to smooth classes. A method of remedying this is to convert proper smooth classes to Fraïssé classes by adding partial functions which map points to sets. The newest EPPA results from \cite{hubička2022eppaclassesstrengtheningsherwiglascar} can handle this type of Fraïssé class and language. However, adding partial functions which map points to sets gives rise to non-standard model theoretic structures. For completeness, we will describe these structures. 
\begin{definition}
    Let $\mathcal{L}$ be a relational language. An \textit{extended language of $\mathcal{L}$} is a language $\mathcal{L}_{ext} = \mathcal{L} \cup \{U\} \cup \mathcal{L}_{fun}$, where $\mathcal{L}_{fun}$ is a set of function symbols and $U$ is a new unary predicate.  For a $\mathcal{L}_{ext}$-structure $\mathcal{A}$,  let  $U_\mathcal{A}$ denote the set of all elements of $\mathcal{A}$ satisfying $U$. A two-sorted $\mathcal{L}_{ext}$-structure $\mathcal{A}$ is an \textit{extended $\mathcal{L}_{ext}$-structure} if  the set $\mathcal{A}- U_\mathcal{A}$ is precisely the power set $\mathcal{P}(U_{\mathcal{A}})$;  for every relation $R\in \mathcal{L}$ of arity $n$, $R^\mathcal{A} \subseteq U_{\mathcal{A}}^n$; and for every function $f\in \mathcal{L}_{fun}$ of arity $n$, $f$ is a partial function such that $dom(f)\subseteq U^n_\mathcal{A}$ and $range(f) \subseteq \mathcal{P}(U_\mathcal{A})$.
\end{definition}
    \begin{definition}
    \deflab{extended}
        An \textit{extended homomorphism} $\Phi: \mathcal{A}\rightarrow \mathcal{B}$ between two extended $\mathcal{L}_{ext}$-structures $\mathcal{A}$ and $\mathcal{B}$ is a function $\Phi_U: U_{\mathcal{A}}\rightarrow U_{\mathcal{B}}$ such that
        for all relations $R\in \mathcal{L}$ $$(a_1,\dots, a_n) \in R^\mathcal{A} \Rightarrow (\Phi_U(a_1),\dots, \Phi_U(a_n))\in R^\mathcal{B}$$
        and for any function $f\in \mathcal{L}_{fun}$, we have that 
        $$f^\mathcal{A}(a_1,\dots, a_n) \subseteq f^\mathcal{B}(\Phi_U(a_1),\dots, \Phi_U(a_n))$$
        $\Phi$ is an \textit{extended embedding of $\mathcal{A}$ into $\mathcal{B}$} if for all $R\in \mathcal{L}$ (including equality) $$(a_1,\dots, a_n) \in R^\mathcal{A} \Leftrightarrow (\Phi_U(a_1),\dots, \Phi_U(a_n))\in R^\mathcal{B}$$
        and for any function $f\in \mathcal{L}_{fun}$, we have that 
        $$f^\mathcal{A}(a_1,\dots, a_n) = f^\mathcal{B}(\Phi_U(a_1),\dots, \Phi_U(a_n))$$
        Whenever an extended embedding $\Phi:\mathcal{A}\rightarrow \mathcal{B}$ is onto, $\Phi$ is an \textit{extended isomorphism} between $\mathcal{A}$ and $\mathcal{B}$. 
        An extended structure $\mathcal{A}$ is an \textit{extended  substructure} of the extended structure $\mathcal{B}$ if $U_{\mathcal{A}}\subseteq U_{\mathcal{B}}$ and the map $id:U_\mathcal{A}\rightarrow U_\mathcal{B}$ is an extended embedding between $\mathcal{A}$ and $\mathcal{B}$. Whenever $\mathcal{A}$ is an extended substructure of $\mathcal{B}$, we will write $\mathcal{A}\preceq \mathcal{B}$. An \textit{extended automorphism } of an extended structure $\mathcal{A}$ is an extended isomorphism $\Phi: \mathcal{A}\rightarrow \mathcal{A}$. An \textit{extended partial automorphism} of an extended structure $\mathcal{A}$ is an extended isomorphism $\Phi: \mathcal{C}\rightarrow \mathcal{D}$ such that $\mathcal{C},\mathcal{D} \preceq \mathcal{A}$.

    \end{definition}

\begin{definition}
Let $(\mathcal{K},\preceq)$ be a class of extended $\mathcal{L}_{ext}$-structures related by $\preceq$. 
\begin{enumerate}
    \item We say $\mathcal{K}$ has closure under extended substructure if $K$ is closed under $\preceq$.
    \item $(\mathcal{K},\preceq)$ is an \textit{disjoint amalgamation class} if $\mathcal{K}$ is closed under extended substructure, and, for any $\mathcal{A},\mathcal{B}, \mathcal{C}\in \mathcal{K}$ such that $U_\mathcal{A} = U_\mathcal{B}\cap U_\mathcal{C}$, $\mathcal{A}\preceq \mathcal{B}, \mathcal{C}$, there is some $\mathcal{D}\in \mathcal{K}$ such that $U_{\mathcal{D}} = U_\mathcal{C}\cup U_\mathcal{B}$ and $\mathcal{B}, \mathcal{A}, \mathcal{C}\preceq \mathcal{D}$.

    $(\mathcal{K},\preceq)$ is a \textit{free amalgamation class} if it is a disjoint amalgamation class and for any $\mathcal{A}$, $\mathcal{B}$, $\mathcal{C}\in \mathcal{K}$ as above, we can choose $\mathcal{D}\in \mathcal{K}$ such that for any tuple $\overline{d}$ which includes elements from both $U_\mathcal{B}-U_\mathcal{A}$ and $U_\mathcal{C}-U_\mathcal{A}$, then no relation $R\in \mathcal{L}$ holds on $\overline{d}$ in $\mathcal{D}$ and $\overline{d}$ is not in the domain of any function $f\in \mathcal{L}_{fun}$ in $\mathcal{D}$. 

\end{enumerate}
\end{definition}

We will now only work with finite permutation groups $\Gamma_{\mathcal{L}}$ on the extended language $\mathcal{L}^F$ which fix  $\mathcal{L}^F-\mathcal{L}$ and preserve equality, arity, and symbol type. The definitions of EPPA, and $\Gamma_\mathcal{L}$-EPPA from the previous section generalize to extended structures and extended partial automorphisms. In this context, we use the following main result from \cite{hubička2022eppaclassesstrengtheningsherwiglascar}:

\begin{theorem}
\thmlab{freeEppa}
Let $\mathcal{L}$ be a countable, relational language, and let $\mathcal{L}_{ext}$ be an extended language of $\mathcal{L}$ with only unary functions. Let $\Gamma_\mathcal{L}$ be a finite permutation group on $\mathcal{L}$, and therefore on $\mathcal{L}_{ext}$. If $(\mathcal{K},\preceq)$ is a free amalgamation class of extended $\mathcal{L}_{ext}$-structures, then $\mathcal{K}$ has $\Gamma_\mathcal{L}$-EPPA.
\end{theorem}
Note that this theorem still applies to standard model-theoretic structures in relational languages, as in these languages, the notion of structures and an extended structure can be viewed identically. The novelty of extended structures is in the advent of functions which map points to sets.

\subsubsection{1-local smooth classes}
In this section, we will prove that 1-local (defined below)  classes have a "Fraïssé-ification" to a class of extended structures, generalizing the argument in \S 6 of \cite{Evans_2019}. We prove that 1-local smooth classes have $\leq$-EPPA and that $\leq$-EPPA is preserved in a number of merges of these classes. 

\begin{definition}
    \deflab{nlocal}
    A smooth class $(K,\leq)$ which is closed under substructure is $1$-local if 
    \begin{enumerate}
        \item $(K,\leq)$ has smooth intersections. From the property of smooth intersections, for any $A, B\in K$ with $A\subseteq B$ there is a unique smallest $C$ such that $A\subseteq C\leq B$. We define $cl_B(A) = C$. 
        \item Whenever $A, C\subseteq B$ are such that if $A,C\leq B$, then $A\cup C\leq B$. 
    \end{enumerate}
\end{definition}
Notice the above properties give us a characterization of the closure relation defined above:
\begin{proposition}
If $(K,\leq)$ is closed under substructure and has smooth intersections, then $(K,\leq)$ satisfies (2) in \defref{nlocal} if and only if for all $A, B\in K$ with $A\subseteq B$, $cl_B(A) = \bigcup_{a\in A} cl_B(\{a\})$.
\end{proposition}
\begin{proof}
$(\Rightarrow):$ 
Let $A\subseteq B \in K$. First, by condition (2) in \defref{nlocal}, and by definition of $cl_B$, $\bigcup_{a\in A} cl_B(\{a\})\leq B$, and thus $cl_B(A) \subseteq \bigcup_{a\in A} cl_B(\{a\})$. Moreover, by definition of $cl_B$, $cl_B(\{a\}) \subseteq cl_B(A)$ for all $a\in A$. 

$(\Leftarrow):$ Suppose $A,C\leq B$. By assumption, we have that $cl_B(A\cup C) = \bigcup_{b\in A\cup C} cl_B(\{b\})$. But, note that $\bigcup_{b\in A\cup C} cl_B(\{b\}) \subseteq A\cup C$. This follows from the fact that since $A\leq B$, for every $b\in A$, $cl_B(\{b\}) = cl_A(\{b\})$. The same is true of $C$. Thus, $A\cup C \leq B$. 
\end{proof}
We also immediately get: 
\begin{proposition}
    If $(K_1,\leq_1)$ and $(K_2,\leq_2)$ are both 1-local, then $(K_1\circledast K_2,\leq_*)$ is 1-local.
\end{proposition}

\begin{definition}
    Let $(K,\leq)$ be a 1-local smooth class in the language $\mathcal{L}$. Let $\Gamma_\mathcal{L}$ be a finite permutation group of $\mathcal{L}$. Call $\Gamma_\mathcal{L}$ \textit{closure-preserving on $K$} if whenever $A\in K$, $\gamma\in \Gamma_\mathcal{L}$ and $f$ is a partial $\gamma$-permorphism of $A$ with domain $D$ and range $R$ such that $D,R\leq A$, then for any $d\in D$, $f[cl_A(\{d\})] = cl_A(\{f(d)\})$.   
\end{definition}
\begin{lemma}
\label{triveppa}
For any 1-local smooth class $(K,\leq)$ in a relational language $\mathcal{L}$, $\Gamma_\mathcal{L} = \{id\}$ is closure-preserving on $K$.
\end{lemma} 
\begin{proof}
    Suppose $A\in K$ and $f$ is a partial automorphism of $A$ with domain $D$ and range $R$ such that $D, R\leq A$. Let $d\in D$. Note that $cl_D(\{d\}) = cl_A(\{d\})$ since $D\leq A$. Now, $f[cl_A(\{d\})] \leq A$ since $cl_A(\{d\})\leq A$ and $f$ is an isomorphism. By definition of $cl_A$, $cl_A(\{f(d)\})\subseteq f[cl_A(\{d\}]$. Moreover, $f^{-1}[cl_A(\{f(d)\})] \leq A$ and by definition, as $d\in f^{-1}[cl_A(\{f(d)\})]$, $cl_A(\{d\})\subseteq f^{-1}[cl_A(\{f(d)\})]$. Thus, $f[cl_A(\{d\})] \subseteq cl_A(\{f(d)\})$. Therefore, $f[cl_A(\{d\})] = cl_A(\{f(d)\})$.   
\end{proof}
We now give a construction used to prove the main result of this section. 
\begin{construction}
\label{construction}
    Let $(K,\leq)$ be a 1-local, smooth class in a relational language $\mathcal{L}$ with dAP. Define $\mathcal{L}^F = \mathcal{L}\cup \{U\}\cup \{f_m: m\geq 1\}$, where each $f_m$ is a unary function. We will consider extended $\mathcal{L}^F$- structures such that $f_k$ is unary and outputs a set of size $k$.  
    
    For each $A\in K$, define the extended expansion $\mathcal{A}^F$ of $A$ so that $U_{\mathcal{A}^F}$ is the universe of $A$, and for all $R\in \mathcal{L}$, and $(a_1,\dots,a_n) \in U_{\mathcal{A}^F}$, 
    $$\mathcal{A}^F\models R(a_1,\dots a_n)\Leftrightarrow A\models R(a_1,\dots, a_n)$$
    Moreover, for every $j\in \omega$, $dom(f^{\mathcal{A}^F}_j) = \{a\in U_{\mathcal{A}^F}: |cl_A(\{a\})| = j\}$ and for $a\in dom(f^{\mathcal{A}^F}_j)$, $f^{\mathcal{A}^F}_j(a) = cl_A(\{a\})$. 
Define $\mathcal{K}^F := \{\mathcal{A}^F: A\in K\}$, and consider the class $(K^F, \preceq)$, where $\preceq$ is the relation defined in \defref{extended}.
\begin{claim}
\label{preceqpreserve}
    $\mathcal{A}^F \preceq \mathcal{B}^F \Leftrightarrow A\leq B$
\end{claim}
\begin{proof}
     If $\mathcal{A}^F$ and $\mathcal{B}^F$ are in $\mathcal{K}^F$ such that $\mathcal{A}^F\preceq \mathcal{B}^F$, then for $A = U_{\mathcal{A}^F}$ and $B =U_ {\mathcal{B}^F}$, we have $A\subseteq B$ as $\mathcal{L}$-structures and $A,B\in K$. Moreover, from the definitions of $\mathcal{A}^F$ and $\mathcal{B}^F$, $$cl_B(A) = \bigcup_{a\in A} cl_B(\{a\}) =  \bigcup_{a\in A}  f^{\mathcal{B}^F}_{|cl_B(\{a\})|}(a) = \bigcup_{a\in A} f^{\mathcal{A}^F}_{|cl_A(\{a\})|}(a)= \bigcup_{a\in A} cl_A(\{a\}) = cl_A(A) = A$$
    Thus, by definition, $A\leq B$.

    On the other hand, if $A, B \in K$ such that $A\leq B$, then $U_{\mathcal{A}^F} \subseteq U_{\mathcal{B}^F}$ and $cl_A(C) = cl_B(C)$ for any $C\subseteq A$. By definition, we get that $f_k^{\mathcal{A}^F} = f_k^{\mathcal{B}^F}$ for all $k\in \omega$. Thus, $\mathcal{A}\preceq \mathcal{B}$. 
\end{proof}
    
   \begin{claim}
       $(\mathcal{K}^F,\preceq)$ is a disjoint amalgamation class
   \end{claim}
   Note that if $(K,\leq)$ has fAP, then the same proof below immediately shows $(K^F,\preceq)$ is a free amalgamation class, as all functions are unary.
    \begin{proof}
    Suppose that $\mathcal{C}\preceq \mathcal{A}^F$ where $\mathcal{A}^F \in \mathcal{K}^F$. We must show there exists some $D\in K$ such that $\mathcal{C} = \mathcal{D}^F$.  Clearly, $U_{\mathcal{C}}$ is an $\mathcal{L}$-substructure of $ U_{\mathcal{A}^F}$. Denote the $\mathcal{L}$-structure $U_{\mathcal{C}}$ as $C$. Let $A$ denote the element of $K$ corresponding to $\mathcal{A}^F$. As $(K,\leq)$ is closed under substructure, $C\in K$. By definition of $\preceq$, for all $c\in C$ with $|cl_A(\{c\})|=k$, $f_k^{\mathcal{A}^F}(c) = f_k^\mathcal{C}(c) = cl_A(\{c\}) \subseteq C$. This implies that $C\leq A$, so $cl_A(\{c\})= cl_C(\{c\})$. For any $n\neq k$, $f_n^{\mathcal{A}^F}$, and therefore $f_n^\mathcal{C}$, is undefined on $a$. By definition, we have that $\mathcal{C} = \mathcal{C}^F$, the extended structure corresponding to $C$. 

    We now show that $(K^F, \preceq)$ has disjoint amalgamation.
   Suppose $\mathcal{A}^F\preceq \mathcal{B}^F$ and $\mathcal{A}^F \preceq \mathcal{C}^F$ with $U_{\mathcal{B}^F} \cap U_{\mathcal{C}^F} = U_{\mathcal{A}^F}$. Then, by Claim \ref{preceqpreserve}, $A\leq B$ and $A\leq C$. As $(K,\leq)$ has dAP, there is some $D\in K$ for which $A, B, C\leq D$ and $D$ is a disjoint amalgam of $C$,$B$ over $A$. By Claim \ref{preceqpreserve}, $\mathcal{A}^F,\mathcal{B}^F, \mathcal{C}^F \preceq \mathcal{D}^F$, and $U_{\mathcal{D}^F}$ is indeed the amalgamation of $U_{\mathcal{B}^F}$ and $U_{\mathcal{C}^F}$ over $U_{\mathcal{A}^F}$. Thus, $(\mathcal{K}^F, \preceq)$ has the amalgamation property. 
    \end{proof}
\end{construction}
Construction \ref{construction} originated from (\cite{Evans_2019}. \cite{Evans_2021}). It is worth noting that any smooth class which satisfies (1) in \defref{nlocal} also corresponds to an extended class using a construction similar to Construction \ref{construction}, but the functions needed for that construction need not be unary. The main magic of 1-local classes is that their closure relations can be described point by point, and thus only require the addition of unary functions. The advantage of unary functions is that they preserve free amalgamation from the smooth class. This allow us to apply \thmref{freeEppa}.

\noindent We now prove the following:
\begin{proposition}
\label{convertEPPA}
    If $(K,\leq)$ is a $1$-local class with dAP in the relational language $\mathcal{L}$, then there is some extended amalgamation class $(\mathcal{K}^F, \preceq)$ in an extended language $\mathcal{L}^F$ of $ \mathcal{L}$ such that for any closure preserving permutation group $\Gamma_\mathcal{L}$ on $(K,\leq)$,  $(K,\leq)$ has $\Gamma_\mathcal{L}$-$\leq$-EPPA if and only if $(\mathcal{K}^F, \preceq)$ has $\Gamma_{\mathcal{L}}$-EPPA
\end{proposition}
\begin{proof}
Let $(\mathcal{K}^F,\preceq)$ be the extended smooth class corresponding to $(K,\leq)$ given by Construction \ref{construction}. 
Assume $(K,\leq)$ has $\Gamma_\mathcal{L}$-$\leq$-EPPA. Note that $\Gamma_\mathcal{L}$ only permutes elements of $\mathcal{L}$, so we may consider it to be a permutation group on $\mathcal{L}^F$. We now prove that $(\mathcal{K}^F,\preceq)$ has $\Gamma_\mathcal{L}$-EPPA. Let $\mathcal{A}^F$ be in $(\mathcal{K}^F, \preceq)$ and $\gamma_1,\dots, \gamma_n\in \Gamma_\mathcal{L}$. Suppose for $i\leq n$, $f_i$ is an extended partial $\gamma_i$-permorphism of $\mathcal{A}^F$. We want to find some $\mathcal{B}^F \in \mathcal{K}^F$ such that for $i\leq n$, there exists an extended $\gamma_i$-permorphism $g_i$ of $\mathcal{B}^F$ such that $g_i\supset f_i$. 

We may consider each extended $\gamma_i$-permorphism $f_i$ as a $\gamma_i$-permorphism $f_i$ of $A$ in the language $\mathcal{L}$, as $dom(f_i), range(f_i)\leq A$ by the construction of $\mathcal{A}^F$. As $(K,\leq)$ has $\Gamma_\mathcal{L}$-EPPA, there is some $B\in K$ such that $A\leq B$ and for $i\leq n$, there exists some $\gamma_i$-permorphism $g_i$ of $B$ such that $g_i\supset \phi_i$. By construction, $\mathcal{A}^F\preceq \mathcal{B}^F$. We now show that each $g_i$ is in fact an extended $\gamma_i$-permorphism of $\mathcal{B}^F$. 
 
 As $g_i$ is a $\gamma_i$-permorphism of $B$, we need only check that $g_i$ preserves the functions of $\mathcal{B}^F$. Let $f_k\in \mathcal{L}^F$ and suppose $a\in \mathcal{B}^F$ is defined on $f_k$. Then, we have $g_i(f_k(a)) = g_i(cl_B(\{a\}))$. As $g_i$ is a $\gamma_i$-permorphism and $\Gamma_\mathcal{L}$ is closure-preserving, $g_i(cl_B(\{a\})) = cl_B(\{g_i(a)\}) = f_k(g_i(a))$. This shows $g_i$ preserves the functions of $\mathcal{B}^F$. Thus, $(\mathcal{K}^F, \preceq)$ has $\Gamma_\mathcal{L}$-EPPA. 
 
 \bigskip
 \noindent The other direction follows similarly. 
\end{proof}

\begin{example}
\label{1localex}
We now give examples of classes which are $1$-local
    \begin{enumerate}
        \item $(K_1,\leq_1)$ where $K_1$ is the set of all graphs in the language $\mathcal{L}_1 :=\{E\}$, and $A\leq_1 B$ if and only if for all $b\in B-A$, $\forall a\in A (\neg E(a,b))$. This class is 1-local. 
        \item $(K_2,\leq_2)$ where $K_2$ is the set of all graphs in the language $\mathcal{L}_2 :=\{E\}$, and $A\leq_2 B$ if and only if for all $b\in B-A$, $\forall a\in A ( E(a,b))$. This class is 1-local. 
        \item $(K_3,\leq_3)$ where $K_3$ is the set of all graphs in the language $\mathcal{L}_3 :=\{R\}$, with $R$ 3-ary, and $A\leq_3 B$ if and only if for all $b_1\in B-A$, $\forall b_2\in B$ $\forall a_1\in A (\neg R(a_1,b_1,b_2))$. This class is 1-local. 
        \item Any merge of $1$-local classes is $1$-local. 
        \item All Fraïssé classes are trivially $1$-local. 
    \end{enumerate}
\end{example}

Thus we can prove:
\begin{proposition}
\thmlab{1local}
    If  a class $(K, \leq)$ is $1$-local and has fAP, then $K$ has $\leq$-EPPA.
\end{proposition}
\begin{proof}
    Putting together Lemma \ref{triveppa} and Proposition \ref{convertEPPA}, $(K, \leq)$ has $\leq$-EPPA if and only if $(\mathcal{K}^F, \preceq)$ has EPPA. Since $(K,\leq)$ has fAP, then $(\mathcal{K}^F, \preceq)$ will have free amalgamation as an extended amalgamation class. By \thmref{freeEppa}, $(\mathcal{K}^F,\preceq)$ has EPPA, and therefore, $(K,\leq)$ has $\leq$-EPPA. 
 \end{proof}
This, in particular, shows that merges of 1-local classes which have fAP will also have $\leq$-EPPA. Let $(K_{LO},\subseteq)$ once again denote the Fraïssé class of all finite linear orders. We now use the theorem given by \cite{Evans_2021} tailored to extended classes to prove the following:
\begin{corollary}
If $(K,\leq)$ is 1-local with free amalgamation, then $(K\circledast K_{LO}, \leq)$ has the Ramsey property. 
\end{corollary}
\begin{proof}
    This follows by converting $(K,\leq)$ to $(\mathcal{K}^F, \preceq)$,  which has fAP and therefore $(\mathcal{K}^F\circledast K_{LO}, \preceq)$ has the Ramsey property by Theorem 1.3 of \cite{Evans_2021}. It will then follow that $(K\circledast K_{LO}, \leq)$ will also have the Ramsey property. 
\end{proof}
These notions of 1-local classes come as a generalization of the arguments given in (\cite{Evans_2019}, \cite{Evans_2021}, \cite{hubička2022eppaclassesstrengtheningsherwiglascar}) that the smooth class of $k$-orientations under successor closures has $\leq$-EPPA. 1-local classes all seem to have a similar definition of the relation $\leq$, which accounts for their closure relations depending on individual points.

\begin{remark}
\label{complement}
It is also possible to apply the arguments above to show that certain 1-local classes which do not have fAP still have $\leq$-EPPA. In particular, the class defined in (2) of Example \ref{1localex} has $\leq$-EPPA, although it certainly does not have fAP. This is done by looking at \textit{complements} of smooth classes. Let $\mathcal{L}$ be a relational language. For an $\mathcal{L}$-structure $A$, let $A^\sim$ denote the structure with the same universe as $A$ such that for every relation $R$, and $\overline{a}\in A^{lg(R)}$, $A\models R(\overline{a})\Leftrightarrow A^\sim \models \neg R(\overline{a})$. For a smooth class $(K,\leq)$, the \textit{complement of $(K,\leq)$} is the class $(K^\sim, \leq_\sim)$ such that $K^\sim = \{A^\sim: A\in K\}$ and $A^\sim \leq_\sim B^\sim \Leftrightarrow A\leq B$. It is easy to see that $(K,\leq)$ is 1-local if and only if $(K^\sim, \leq_\sim)$ is 1-local.  It is also immediate that $(K,\leq)$ has $\leq$-EPPA if and only if $(K^\sim, \leq_\sim)$ has $\leq_\sim$-EPPA. Moreover, if $(K_1,\leq_1)$ is a smooth class, then $(K\circledast K_1,\leq_*)$ has $\leq_*$-EPPA if and only if the merged class $(K^\sim\circledast K_1, \leq_{*})$ has $\leq_{*}$-EPPA.

\end{remark}
\subsubsection{Classes without EPPA via 1-local Smooth Classes}
Studying smooth classes in the context of EPPA properties can also give some results for EPPA properties of true Fraïssé classes. 
\begin{definition}
    Let $\mathcal{L}$ be a countable, relational language. For $F$ a family of finite $\mathcal{L}$-structures, define $Forb_m(F)$ as the class of all $\mathcal{L}$-structures $A$ so that there does not exist a $B\in F$ and a 1-1 mapping $f: B\rightarrow A$ such that for all $R\in \mathcal{L} $, $B\models R(b_1,\dots, b_n) \Rightarrow A \models R(f(b_1),\dots, f(b_n))$. Define $Forb_e(F)$ as the class of all $\mathcal{L}$-structures $A$ so that there does not exist a $B\in F$ and an $\mathcal{L}$-embedding $f: B\rightarrow A$. 

    We say a $\mathcal{L}$-structure is \textit{irreducible} if it cannot be written as a free amalgamation of two of its proper subsets. Let $Forb_{he}(F)$ be the class of all $\mathcal{L}$-structures $A$ so that there does not exist a $B\in F$ and a mapping $f: B\rightarrow A$ such that $f$ is injective on irreducible subsets $B'\subseteq B$ and for all $R\in \mathcal{L} $, $B\models R(b_1,\dots, b_n) \Rightarrow A \models R(f(b_1),\dots, f(b_n))$. 
\end{definition}
It is known by \cite{8d041ed3-5ed1-3e56-9a99-3f73c69d33f4} that if $Forb_e(F)$ is a Fraïssé class with fAP, then $Forb_e(F)$ has EPPA. 

Suppose that $F$ is finite and for every  $A\in Forb_{he}(F)$, there is some infinite $\mathcal{L}$-structure $M$ such that $A\subseteq M$ and every partial automorphism of $A$ extends to an automorphism of M, and, moreover, for every $B\subseteq M$, $B\in Forb_{he}(F)$. By \cite{hubička2022eppaclassesstrengtheningsherwiglascar}, $Forb_{he}(F)$ has EPPA. 

\bigskip
Using examples of smooth classes, we can give a proof using a 1-local class that the assumptions above cannot be dropped completely. 

\begin{proposition}
    There exists a finite, relational language $\mathcal{L}$ and a finite family $F$ of $\mathcal{L}$-structures for which $Forb_e(F)$, $Forb_{he}(F)$, and $Forb_m(F)$ do not have EPPA. 
\end{proposition}
\begin{proof}
\label{noeppa}
    Let $\mathcal{L} = \{E\}$, $E$ a binary relation. Let $K$ be the class of all finite $\mathcal{L}$-structures. Define the smooth class $(K,\leq)$ as $A\leq B$ if and only if $\forall b\in B-A \forall a\in A (\neg E(a,b))$.  

    Fix $A\in K$ such that $A$ contains two points $a_1, a_2$ so that $A\models E(a_1,a_2)$ and a point $a_3$ which is not $E$-connected to any other point in $A$, including itself. \\
    \textbf{Observation}:  Consider the partial automorphism $f:\{a_1\} \rightarrow \{a_3\}$. Suppose there were some $B\in K$ with $A\leq B$ for which there exists some automorphism $g: B\rightarrow B$ extending $f$. As $E(a_1,a_2)$ holds, but $E(a_3, a)$ fails for every $a\in A$, there must be some $b\in B-A$ such that $B\models E(a_3, b)$. But this contradicts that $A\leq B$. Thus, there is no $B\in K$ with $A\leq B$ for which there is an automorphism $g$ of $B$ extending $f$.\\  
    
    Let $\mathcal{L}_R = \{R\}$ and $\mathcal{L}^\# = \mathcal{L}\cup \mathcal{L}_R$, where $R$ is a binary relation. Extend $A$ to a $\mathcal{L}^\#$-structure $A^\#$ by stipulating that $R$ holds on every 2-tuple of $A$. Let $K_R$ be the class of all finite $\mathcal{L}_R$-structures, and let $K^\# = K \circledast K_R$. 
Define $$F^\# := \{C\in K^\#: C \text{ is a one-point extension of $A^\#$ so that } A^\#|_{\mathcal{L}}\not\leq C|_{\mathcal{L}}\}$$
We prove that $Forb_{he}(F^\#)$ does not have EPPA. The results for $Forb_{m}(F^\#)$ and $Forb_e(F^\#)$ follow by a similar argument.

Suppose $Forb_{he}(F^\#)$ did have EPPA. Note that $A^\# \in Forb_{he}(F^\#)$, as all $C\in F^\#$ are irreducible by definition of $\leq$. Then, there would be some $B^\#\in Forb_{he}(F^\#)$ such that for any partial $\mathcal{L}^\#$-automorphisms $f_1,\dots, f_n$ of $A^\#$ (\textbf{Note:} Here, the domain and range of these partial automorphisms are unrestricted), there exist $g_1,\dots, g_n\in Aut(B^\#)$ so that $g_i\supset f_i$ for all $i$.

Notice that we necessarily have that $A = A^\#|_{\mathcal{L}}\leq B^\#|_{\mathcal{L}}$ by the definition of $\leq$ and $F^\#$. Moreover, every partial automorphism of $A$ is a partial automorphism of $A^\#$. Hence, for $B = B^\#|_{\mathcal{L}}$, $A\leq B$ and $B$ is an EPPA witness for $A$ in the language $\mathcal{L}$.  This is a contradiction to the observation.
\end{proof}
In general, smooth classes give rise to a large variety of classes which cannot have EPPA in the Fraïssé sense.
\begin{fact}
    For a smooth class $(K,\leq)$ with a generic $M$, for $A\in K$, let $K_A$ be the class of all $B\in K$ such that $A\leq B$. If $K_A$ has EPPA, then every partial automorphism of $A$ (where the domain and range are simply subsets of $A$) extends to an automorphism of $M$. If for every $A\in K$, $K_A$ has EPPA, then $M$ is the Fraïssé limit of $(K^+,\subseteq)$ where $K^+$ is the closure of $K$ under substructure. 
\end{fact}
Thus, if the generic $M$ of $(K,\leq)$ does not extend every partial automorphism of $A\in K$, then $K_A$ does not have EPPA, and moreover, no subclass of $K_A$ has EPPA. 
\subsection{Merging with 1-local classes}
\label{mergeeppa}
We are now interested in merging classes with EPPA and investigating the resulting merged class. We first observe the following easy fact:
\begin{fact}
    Suppose for every $n\in \omega$, $(K_1,\leq_1)$ contains a structure $A\in K_1$ for which for all $R\in \mathcal{L}$, $R$ holds on every tuple of $A$, and for every $A_0\subseteq A$, $A_0\leq_1 A$. If $(K_2,\leq_2)$ is another smooth class such that $(K_1\circledast K_2,\leq_*)$ has $\leq_*$-EPPA, then $(K_2,\leq_2)$ has $\leq_2$-EPPA. 
\end{fact}
The Fraïssé class of equivalence relations under substructure, $(K_E,\subseteq)$, for example, would be an example of a $(K_1,\leq_1)$ as in the above fact. Merges of such classes cannot have any "surprise" EPPA results. For example, any merge $(K^*,\leq_*)$ of $(K_E,\subseteq)$ with a class $(K_2,\leq_2)$ which does not have $\leq_2$-EPPA will not have $\leq_*$-EPPA. 

\bigskip
We now turn to prove some results about classes we know have an EPPA property. 

\bigskip
\textbf{Notation:} For every $n\in \omega$, call $E_n$ an $n$-tuple equivalence relation if $E_n$ is a $2n$-ary relation which is transitive, symmetric, irreflexive on $n$-tuples, and whenever $E_n(\overline{a},\overline{b})$ holds on $n$ tuples $\overline{a}$ and $\overline{b}$, then for any permutation $\overline{b}_\sigma$ of the tuple $\overline{b}$, $E(\overline{a}, \overline{b}_\sigma) $ holds. Let $(K_{E_n},\subseteq)$ denote the class of finite structures with an $n$-tuple equivalence relation in the language $\mathcal{L}_{E_n} = \{E_n\}$. Set $\mathcal{L}_{E_\omega}= \bigcup \{\mathcal{L}_{E_n}: n\in \omega\}$.

We have the following result of  \cite{https://doi.org/10.1002/malq.201400036}:
\begin{theorem}
\thmlab{ivanov} 
    The Fraïssé class $K_{E_\omega} := \circledast_{i\in\omega} K_{E_n}$ in the language $\mathcal{L}_{E_\omega}$ has EPPA. 
\end{theorem}

The fact that the class $K_{E_\omega}$ has EPPA  is interestingly not immediately covered by basic EPPA results that use free amalgamation or forbidden classes. The proof of \thmref{ivanov} proceeds by changing the language, extending to a merge of classes, proving the merge has a permorphism EPPA property, then proving the restriction back to $K_{E_\omega}$ has EPPA. It is reminiscent of the operation we applied to $1$-local smooth classes to attain a class of extended structures, though here the goal is to gain a class with free amalgamation. At its core, the proof boils down to an application of \thmref{freeEppa}. Using this, we prove a generalization:

\begin{theorem}
\thmlab{ivanovVar}
Let $(K_{E_\omega}, \subseteq)$ be the Fraïssé class defined above. Let $(K_2,\leq_2)$ be any $1$-local smooth class with fAP in a countable relational language. Then, the merge $K_0 := K_{E_\omega} \circledast K_2$ in the language $\mathcal{L}_0:= \mathcal{L}_{E_\omega}\cup \mathcal{L}_2$ has $\leq_*$-EPPA.
\end{theorem}
\begin{proof}
Note that, in this merge, $A\leq_* B \Leftrightarrow A\leq_2 B \;\& \; A|_{\mathcal{L}_{E_{\omega}}}\subseteq B|_{\mathcal{L}_{E_\omega}}$. 

Define a new language $\mathcal{L}^\# = \{P_{n,i}: i,n\in \omega\}$ where $P_{n,i}$ is an $n$-ary relation, and let $K^\#$ be all $\mathcal{L}^\#$ structures such that $P_{n,i}\cap P_{n,j} = \emptyset$ whenever $i\neq j$; $P_{n,i}(a_1,\dots, a_n) \Leftrightarrow P_{n,i}(a_{\sigma(1)},\dots, a_{\sigma(n)})$ for any $\sigma\in Sym(n)$; and every $P_{n,i}$ is an irreflexive relation (i.e., $P_{n,i}$ does not hold on a tuple with repeated elements). 

Now, let $K^* = K^\#\circledast K_2$, and let $\mathcal{L}^* = \mathcal{L}^\#\cup \mathcal{L}_2$. 
Let $\mathcal{L}'$ be any set such that $\mathcal{L}_2\subseteq \mathcal{L}'\subseteq \mathcal{L}^*$ and $\mathcal{L}'- \mathcal{L}_2$ is finite. Define $K(\mathcal{L}')$ to be all $\mathcal{L}'$ structures which are in $K^*$.  Let $\Gamma_{\mathcal{L}'}$ be all arity-preserving permutations of the language $\mathcal{L}'-\mathcal{L}_2$. The first observation we make is that $K(\mathcal{L}')$ has free amalgamation with respect to $\leq_*$. Since every $\gamma\in \Gamma_{\mathcal{L}'}$ fixes $\mathcal{L}_2$, by \thmref{freeEppa} and Theorem \ref{convertEPPA}, $K(\mathcal{L}')$ has $\Gamma_{\mathcal{L}'}$-$\leq_*$-EPPA. 

\bigskip
We now choose any $A_0\in K_0$ and any partial automorphisms $f_1,\dots, f_n$ of $A_0$ for which $dom(f_i)$ and $range(f_i) \leq_* A_0$ for $i\leq n$. There is a maximum $q$ for which $E_q$ is defined on $A_0$. Thus, for $m\leq q$ we can enumerate all the classes of each $E_m$ which is satisfied on $A_0$. Suppose for each $m\leq q$, $E_m$ has $j_m$ many classes in $A$. We will  then set $\mathcal{L}' := \mathcal{L}_2\cup \{P_{m,r}: m\leq q, \; r\leq j_m\}$. 

On the universe of $A_0$, we define an $\mathcal{L}'$-structure $A$ so that $A \models P_{k,i}(a_1,\dots, a_k)$ if and only if $(a_1,\dots, a_k)$ is in the $i$th equivalence class of $E_k$ in $A_0$. We also require that $A|_{\mathcal{L}_2}$ is precisely $A_0|_{\mathcal{L}_2}$. Notice that for each $i$, $f_i$ is a $\gamma$-permorphism of $A$ for some $\gamma\in \Gamma_{\mathcal{L}'}$ and permutes only symbols in $\mathcal{L}' - \mathcal{L}_2$.

We automatically get that $A \in K(\mathcal{L}')$. $A$ has a $\Gamma_{\mathcal{L}'}$-$\leq_*$- EPPA witness $B\in K(\mathcal{L}')$ such that $A\leq_* B$ and there exist $g_1,\dots, g_n$ permorphisms of $B$ such that $g_i\supset f_i$. For each $\ell \leq |B|$, we can find some $r\in\omega$ such that $P_{\ell, r}$ is an $\mathcal{L}^*$ - relation not yet appearing in $\mathcal{L}'$. We will expand $B$ to a $\mathcal{L}'\cup \{P_{\ell, r}: \ell\leq |B|\}$ structure as follows:

For any irreflexive tuple $(b_1,\dots, b_\ell)$ of $B$, if $(b_1,\dots, b_\ell)$ does not hold on any relation in $\mathcal{L}'$, then we require that $P_{\ell,r}(b_1,\dots, b_\ell)$ holds in the expansion of $B$, call it $B'$. Suppose $g$ is a permorphism of $B$. Notice that because all tuples which did not belong to a relation in $B$ now belong to the same relation in $B'$, and since, for each $\ell$ and $r$, $g$ must fix $P_{\ell,r}$, $g$ must be a permorphism of $B'$ as well.  

\bigskip
We will define an $\mathcal{L}_0$-structure $B_0$ with the same universe as $B'$ such that for all $i$, and any $i$-ary tuples $\overline{a}$, $\overline{b}$ from $B'$, $B_0 \models E_i(\overline{a}, \overline{b})$ if and only if $B'\models P_{i,j}(\overline{a}) \land P_{i,j}(\overline{b})$ for some $j$. Moreover,  we set $B_0|_{\mathcal{L}_2} = B|_{\mathcal{L}_2}$. As every tuple of $B_0$ satisfies an appropriate equivalence relation, $B_0\in K_0$. 

Because we have preserved the $\mathcal{L}_2$-structure in our construction, $A_0 \leq_2 B_0$ and $A_0\subseteq B_0$ with respect to the language $\mathcal{L}_{E_\omega}$. Thus, $A_0 \leq_* B_0$. Now, any permorphism $g_i$ of $B'$ is clearly an automorphism of $B_0$, as 
$$B_0 \models E_n(g_i(\overline{b}), g_i(\overline{a})) \Leftrightarrow  \exists t \; B'\models P_{n,t}(g_i(\overline{a}))\land P_{n,t}( g_i(\overline{b})) \Leftrightarrow \exists  w \; B'\models P_{n,w}(\overline{a})\land P_{n,w}(\overline{b}) \Leftrightarrow B_0\models E_n(\overline{a}, \overline{b})$$
And, moreover, it is easy to see that $g_i\supseteq f_i$, the original partial automorphisms of $A_0$. Thus, $B_0$ is a $\leq_*$-EPPA witness for $A_0$.
\end{proof}

Using the argument in Remark \ref{complement}
\begin{corollary}
Let $(K_{E_\omega}, \subseteq)$ be the Fraïssé class defined above. Suppose $(K_2,\leq_2)$ is a (possibly infinite) merge of 1-local classes with fAP and/or their complements, then $(K_{E_\omega} \circledast K_2, \leq_*)$ has $\leq_*$-EPPA.
\end{corollary}
We can generalize the proof above slightly:
\begin{definition}
For $K_{E_\omega}$ defined as above, call a smooth class $(K_{E_\omega}, \leq_1)$ \textit{separably 1-local} if there exists some partition $\omega = I\sqcup J$ where for all $A, B\in \circledast_{j\in J} K_{E_j}$, $A\leq_1 B \Leftrightarrow A\subseteq B$,  and on the class $\circledast_{i\in I} K_{E_i}$, $\leq_1$ is a 1-local relation. 
\end{definition}
We can apply the argument in the proof above to a separably 1-local class $(K_{E_\omega}, \leq_1)$ which has fAP by cutting out $\circledast_{i\in I} K_{E_i}$ and merging it in with a 1-local smooth class $(K_2,\leq_2)$ to get 
\begin{proposition}
For a smooth class separably 1-local class $(K_{E_\omega}, \leq_1)$ with fAP and any 1-local smooth class $(K_2, \leq_2)$ with fAP, $(K_{E_\omega}\circledast K_2, \leq_*)$ has $\leq_*$-EPPA.
\end{proposition}
An easy example of a separably 1-local class $(K_{E_\omega}, \leq_1)$ is to define for $A, B\in K_{E_\omega}$ $$A\leq_1 B \Leftrightarrow \text{for $i>1$, } A|_{\mathcal{L}_{E_i}} \subseteq B|_{\mathcal{L}_{E_i}} \; \& \; \forall b\in B-A \;\forall a\in A (\neg E_1(a,b))$$

\BibTexMode{%
   \bibliographystyle{alpha}
   \bibliography{template}

@article{KL,
author = {D. W. Kueker and M. C. Laskowski},
title = {{On generic structures}},
volume = {33},
journal = {Notre Dame Journal of Formal Logic},
number = {2},
publisher = {Duke University Press},
pages = {175 -- 183},
year = {1992},
doi = {10.1305/ndjfl/1093636094},
URL = {https://doi.org/10.1305/ndjfl/1093636094}
}

@article{NESETRIL1983183,
title = {Ramsey classes of set systems},
journal = {Journal of Combinatorial Theory, Series A},
volume = {34},
number = {2},
pages = {183-201},
year = {1983},
issn = {0097-3165},
doi = {https://doi.org/10.1016/0097-3165(83)90055-9},
url = {https://www.sciencedirect.com/science/article/pii/0097316583900559},
author = {Jaroslav Nešetřil and Vojtěch Rödl}
}

@article{8d041ed3-5ed1-3e56-9a99-3f73c69d33f4,
 ISSN = {10798986},
 URL = {http://www.jstor.org/stable/3109885},
 author = {Ian Hodkinson and Martin Otto},
 journal = {The Bulletin of Symbolic Logic},
 number = {3},
 pages = {387--405},
 publisher = {[Association for Symbolic Logic, Cambridge University Press]},
 title = {Finite Conformal Hypergraph Covers and {G}aifman Cliques in Finite Structures},
 urldate = {2024-10-30},
 volume = {9},
 year = {2003}
}

@article{Evans_2021,
   title={Ramsey properties and extending partial automorphisms for classes of finite structures},
   volume={253},
   ISSN={1730-6329},
   url={http://dx.doi.org/10.4064/fm560-8-2020},
   DOI={10.4064/fm560-8-2020},
   number={2},
   journal={Fundamenta Mathematicae},
   publisher={Institute of Mathematics, Polish Academy of Sciences},
   author={Evans, David M. and Hubička, Jan and Nešetřil, Jaroslav},
   year={2021},
   pages={121–153} }

@article{Evans_2019,
   title={Automorphism groups and {R}amsey properties of sparse graphs},
   volume={119},
   ISSN={1460-244X},
   url={http://dx.doi.org/10.1112/plms.12238},
   DOI={10.1112/plms.12238},
   number={2},
   journal={Proceedings of the London Mathematical Society},
   publisher={Wiley},
   author={Evans, David M. and Hubička, Jan and Nešetřil, Jaroslav},
   year={2019},
   month=mar, pages={515–546} }

@article{BALDWIN19961,
title = {Stable generic structures},
journal = {Annals of Pure and Applied Logic},
volume = {79},
number = {1},
pages = {1-35},
year = {1996},
issn = {0168-0072},
doi = {https://doi.org/10.1016/0168-0072(95)00027-5},
url = {https://www.sciencedirect.com/science/article/pii/0168007295000275},
author = {John T. Baldwin and Niandong Shi}
}

@article{gunatilleka2018theoriesbaldwinshihypergraphsatomic,
      title={The theories of {B}aldwin-{S}hi hypergraphs and their atomic models}, 
      author={Danul K. Gunatilleka},
      year={2018},
      eprint={1803.01831},
      archivePrefix={arXiv},
      primaryClass={math.LO},
      url={https://arxiv.org/abs/1803.01831}, 
}

@article{bodirsky2012newramseyclassesold,
      title={New {R}amsey Classes from Old}, 
      author={Manuel Bodirsky},
      year={2012},
      eprint={1204.3258},
      archivePrefix={arXiv},
      primaryClass={math.LO},
      url={https://arxiv.org/abs/1204.3258}, 
}

@article{kechris2004fraisselimitsramseytheory,
      title={Fraisse Limits, {R}amsey Theory, and Topological Dynamics of Automorphism Groups}, 
      author={A. S. Kechris and V. G. Pestov and S. Todorcevic},
      year={2004},
      eprint={math/0305241},
      archivePrefix={arXiv},
      primaryClass={math.LO},
      url={https://arxiv.org/abs/math/0305241}, 
}

@article{ghadernezhad2016automorphismgroupsgenericstructures,
      title={Automorphism Groups of Generic Structures: Extreme Amenability and Amenability}, 
      author={Zaniar Ghadernezhad and Hamed Khalilian and Massoud Pourmahdian},
      year={2016},
      eprint={1508.04628},
      archivePrefix={arXiv},
      primaryClass={math.LO},
      url={https://arxiv.org/abs/1508.04628}, 
}

@article{axiomatization,
author = {Laskowski, Michael},
year = {2007},
month = {01},
pages = {157-186},
title = {A simpler axiomatization of the {S}helah-{S}pencer almost sure theory},
volume = {161},
journal = {Israel Journal of Mathematics},
doi = {10.1007/s11856-007-0077-8}
}

@article{kechris2006turbulenceamalgamationgenericautomorphisms,
      title={Turbulence, amalgamation and generic automorphisms of homogeneous structures}, 
      author={Alexander S. Kechris and Christian Rosendal},
      year={2006},
      eprint={math/0409567},
      archivePrefix={arXiv},
      primaryClass={math.LO},
      url={https://arxiv.org/abs/math/0409567}, 
}

@article{hubička2022eppaclassesstrengtheningsherwiglascar,
      title={All those {EPPA} classes (Strengthenings of the {H}erwig-{L}ascar theorem)}, 
      author={Jan Hubička and Matěj Konečný and Jaroslav Nešetřil},
      year={2022},
      eprint={1902.03855},
      archivePrefix={arXiv},
      primaryClass={math.CO},
      url={https://arxiv.org/abs/1902.03855}, 
}

@article{https://doi.org/10.1002/malq.201400036,
author = {Ivanov, Aleksander},
title = {An {$\omega$}-categorical structure with amenable automorphism group},
journal = {Mathematical Logic Quarterly},
volume = {61},
number = {4-5},
pages = {307-314},
doi = {https://doi.org/10.1002/malq.201400036},
url = {https://onlinelibrary.wiley.com/doi/abs/10.1002/malq.201400036},
eprint = {https://onlinelibrary.wiley.com/doi/pdf/10.1002/malq.201400036},
year = {2015}
}

@article{Herwig1995ExtendingPI,
  title={Extending partial isomorphisms on finite structures},
  author={Bernhard Herwig},
  journal={Combinatorica},
  year={1995},
  volume={15},
  pages={365-371},
  url={https://api.semanticscholar.org/CorpusID:13082360}
}
}%
\BibLatexMode{\printbibliography}

\end{document}